\documentclass[11pt]{amsart}
\usepackage{amsfonts}
\usepackage{amssymb}
\usepackage{graphicx}
\usepackage{lipsum} 
\usepackage[title]{appendix}
\usepackage{amsmath}
\usepackage{mathrsfs}
\usepackage{hyperref}
\usepackage{color}\definecolor{Red}{cmyk}{0,1,1,0}

\numberwithin{equation}{section}

\usepackage{tikz, tkz-euclide}
\usetikzlibrary{calc}
\usetikzlibrary{decorations.markings}

\usepackage{bbm}
  \newtheorem*{theorem*}        {Theorem}
	\newtheorem*{conjecture*}   {Conjecture}

  \newtheorem*{lemma*}          {Lemma}
    \newtheorem*{claim*}          {Claim}

    \newtheorem{teo}{Theorem}
  \newtheorem{coro}[teo]{Corollary}
  
  \newtheorem{prop}[teo]{Proposition}
  \newtheorem{defi}[teo]{Definition}
  \newtheorem{note}[teo]{Remark}
  \newtheorem{lema}[teo]{Lemma}
  
  \newtheorem{exem}[teo]{Example}

\begin{document}
%\pagestyle{headings}

%%%%%%%%%%%%%%%%%       FORMATO

%\magnification=\magstep1\hoffset=0.cm
\voffset=-1.5truecm\hsize=16.5truecm    \vsize=24.truecm
\baselineskip=14pt plus0.1pt minus0.1pt \parindent=12pt
\lineskip=4pt\lineskiplimit=0.1pt      \parskip=0.1pt plus1pt

\def\ds{\displaystyle}\def\st{\scriptstyle}\def\sst{\scriptscriptstyle}

\global\newcount\numsec\global\newcount\numfor
\gdef\profonditastruttura{\dp\strutbox}
\def\senondefinito#1{\expandafter\ifx\csname#1\endcsname\relax}
\def\SIA #1,#2,#3 {\senondefinito{#1#2}
\expandafter\xdef\csname #1#2\endcsname{#3} \else
\write16{???? il simbolo #2 e' gia' stato definito !!!!} \fi}
\def\etichetta(#1){(\veroparagrafo.\veraformula)
\SIA e,#1,(\veroparagrafo.\veraformula)
 \global\advance\numfor by 1
% \write15{@def@equ(#1){\equ(#1)} \%:: ha simbolo= #1 }
 \write16{ EQ \equ(#1) ha simbolo #1 }}
\def\etichettaa(#1){(A\veroparagrafo.\veraformula)
 \SIA e,#1,(A\veroparagrafo.\veraformula)
 \global\advance\numfor by 1\write16{ EQ \equ(#1) ha simbolo #1 }}
\def\BOZZA{\def\alato(##1){
 {\vtop to \profonditastruttura{\baselineskip
 \profonditastruttura\vss
 \rlap{\kern-\hsize\kern-1.2truecm{$\scriptstyle##1$}}}}}}
\def\alato(#1){}
\def\veroparagrafo{\number\numsec}\def\veraformula{\number\numfor}
\def\Eq(#1){\eqno{\etichetta(#1)\alato(#1)}}
\def\eq(#1){\etichetta(#1)\alato(#1)}
\def\Eqa(#1){\eqno{\etichettaa(#1)\alato(#1)}}
\def\eqa(#1){\etichettaa(#1)\alato(#1)}
\def\equ(#1){\senondefinito{e#1}$\clubsuit$#1\else\csname e#1\endcsname\fi}
\let\EQ=\Eq

%%%%%%%%%%%%%%% Definitions

\def\\{\noindent}
\def\v{\vskip.1cm}
\def\vv{\vskip.2cm}

%\def\tende#1{\vtop{\ialign{##\crcr\rightarrowfill\crcr
%              \noalign{\kern-1pt\nointerlineskip}
%              \hskip3.pt${\scriptstyle #1}$\hskip3.pt\crcr}}}
%\def\otto{{\kern-1.truept\leftarrow\kern-5.truept\to\kern-1.truept}}
%\def\arm{{}}
%\font\bigfnt=cmbx10 scaled\magstep1

%
  \def\P{\mathop{\textrm{\rm P}}\nolimits}                  %P
  \def\d{\mathop{\textrm{\rm d}}\nolimits}                  %d
  \def\exp{\mathop{\textrm{\rm exp}}\nolimits}              %exp
	\def\supp{\mathop{\textrm{\rm supp}}\nolimits}            %supp
	\def\Int{\mathop{\textrm{\rm Int}}\nolimits}            %Int
	\def\Ext{\mathop{\textrm{\rm Ext}}\nolimits}            %Ext
	\def\id{\mathop{\textrm{\rm id}}\nolimits}            %identity
    \def\sf{\mathop{\textrm{\rm sf}}\nolimits}            %sf
    \def\Var{\mathop{\textrm{\rm Var}}\nolimits}            %Variação
    
    \newcommand{\bbr}{\ensuremath        {\mathbb{R}}}
\newcommand{\bbz}{\ensuremath        {\mathbb{Z}}}
\newcommand{\bbn}{\ensuremath        {\mathbb{N}}}
\newcommand{\bbc}{\ensuremath        {\mathbb{C}}}
\newcommand{\caln}{\ensuremath        {\mathcal{N}}}
\newcommand{\cc}{\ensuremath        {\circ}}
\newcommand{\calb}{\ensuremath        {\mathcal{B}}}
\newcommand{\calf}{\ensuremath        {\mathcal{F}}}
\newcommand{\calc}{\ensuremath        {\mathcal{C}}}
\newcommand{\calr}{\ensuremath        {\mathcal{R}}}
\newcommand{\call}{\ensuremath        {\mathcal{L}}}
\newcommand{\ap}{\ensuremath        {\alpha}}
\newcommand{\bt}{\ensuremath        {\beta}}
\newcommand{\tm}{\ensuremath        {\times}}
\newcommand{\D}{\ensuremath        {\Delta}}
\newcommand{\oo}{\ensuremath        {\overline}}
\newcommand{\bbe}{\ensuremath        {\mathbb{E}}}
\let\a=\alpha
\let\b=\beta
\let\e=\varepsilon
\let\f=\varphi
\let\g=\gamma
\let\h=\eta
\let\k=\kappa
\let\l=\lambda
\let\m=\mu
\let\n=\nu
\let\o=\omega
\let\p=\pi
\let\ph=\varphi
\let\r=\rho
\let\s=\sigma
\let\t=\tau
\let\th=\vartheta
\let\y=\upsilon
\let\x=\xi
\let\z=\zeta
\let\D=\Delta
\let\F=\Phi
\let\G=\Gamma
\let\L=\Lambda
\let\Th=\Theta
\let\O=\Omega
\let\P=\Pi
\let\Ps=\Psi
\let\Si=\Sigma
\let\x=\xi
\let\Y=\Upsilon

    \newcommand\bfblue[1]{\textcolor{blue}{\textbf{#1}}}
\newcommand\blue[1]{\textcolor{blue}{}}

\thispagestyle{empty}

\begin{center}
{\LARGE Infinite DLR Measures and Volume-Type Phase Transitions on Countable Markov Shifts}
\vskip.5cm
Elmer R. Beltr\'{a}n$^{1,3}$, Rodrigo Bissacot$^{1}$, Eric O. Endo$^{2}$
\vskip.3cm
\begin{footnotesize}
$^{1}$Institute of Mathematics and Statistics (IME-USP), University of S\~{a}o Paulo, Brazil\\
$^{2}$NYU-ECNU Institute of Mathematical Sciences at NYU Shanghai, 3663 Zhongshan Road North, Shanghai, 200062, China\\
$^{3}$Universidad Cat\'olica del Norte, Departamento de Matem\'aticas, Avenida Angamos 0610, Antofagasta - Chile
\end{footnotesize}
\vskip.1cm
\begin{scriptsize}
emails: rusbert.unt@gmail.com; rodrigo.bissacot@gmail.com; ericossamiendo@gmail.com
\end{scriptsize}

\end{center}

%---------- Definir equacao
\def\be{\begin{equation}}
\def\ee{\end{equation}}

\vskip1.0cm
%\begin{abstract}
\begin{quote}
{\small

\textbf{Abstract.} \begin{footnotesize} We consider the natural definition of DLR measure in the setting of $\sigma$-finite measures on countable Markov shifts. We prove that the set of DLR measures contains the set of conformal measures associated with Walters potentials. In the BIP case, or when the potential normalizes the Ruelle's operator, we prove that the notions of DLR and conformal coincide. On the standard renewal shift, we study the problem of describing the cases when the set of the eigenmeasures jumps from finite to infinite measures when we consider high and low temperatures, respectively. For this particular shift, we prove that there always exist finite DLR measures, and we have an expression to the critical temperature for this volume-type phase transition, which occurs only for potentials with the infinite first variation.
 \end{footnotesize}}
\end{quote}



\section{Introduction}
\noindent

One of the main objects in equilibrium statistical mechanics, in terms of measures, is the notion of \textit{DLR measure}, which in the probability and mathematical physics communities is synonymous with \textit{Gibbs measure}. The name is in honor of R. Dobrushin \cite{Do1, Do2, Do3}, O. Lanford, and D. Ruelle \cite{LaRu}, who introduced a system of equations involving conditional expectations which characterize the DLR measures, now called \textit{DLR equations}, see \cite{FV, Geo, RaSe}.

We avoid the name Gibbs measure because it is used with several different meanings by the ergodic theory and dynamical systems communities, which sometimes coincide with the notion of DLR measure, but in some cases not. There exist several different notions of Gibbs measures used by dynamicists in addition to DLR measures, some examples are conformal measures \cite{DeUr}, Gibbs measures in the sense of Capocaccia \cite{Capo}, $g$-measures \cite{Kea}, eigenmeasures (associated to the Ruelle operator) \cite{Bo, Ru2}, equilibrium measures, and many other notions. See \cite{Ke, Ki} for positive results when the lattice is $\mathbb{Z}^d$, where the authors study when some of these notions coincide and alphabet $S$ (state space) is finite, see \cite{Mu, Mu1} for the case where $S=\mathbb{N}$. 

For finite state space Markov shifts contained in $S^{\mathbb{N}}$ the Ruelle-Perron-Frobenius Theorem, due to Ruelle \cite{Ru, Ru2}, guarantees the existence of conformal probability measures when the potential belongs to the Walters class. It is known that conformal measures and eigenmeasures are equivalent \cite{ANS, Sa5}. Moreover, we also know that DLR measures and eigenmeasures are equivalent notions even for continuous potentials, see \cite{CLS}. On the other hand, in $S^{\mathbb{Z}}$, there are examples of $g$-measures which are not DLR measures \cite{FGG} and examples of DLR measures which are not $g$-measures \cite{BEvEL}, for characterization when these two notions coincide see \cite{BFV}. Nowadays, after Ruelle \cite{Ru2} and Bowen \cite{Bo}, the study of these measures and their properties is inside of a class of results in ergodic theory called \textit{Thermodynamic Formalism}.

In the last two decades, the theory was extended to the non-compact state space $S=\mathbb{N}$ by several authors \cite{FFY, MaUr, Sa1, Sa5}. In particular, O. Sarig produced a good amount of results with applications in dynamic systems and used the powerful analytical tool of the Ruelle operator. A helpful review of Sarig's contributions is given by Y. Pesin in \cite{Pe}.

In this paper we are focused on countable Markov shifts $\Sigma_A \subseteq\bbn^{\bbn \cup \{0\}}$. When a Markov shift satisfies the BIP property, we can show that both notions conformal and DLR coincide. Apart from these shifts, there exist countable Markov shifts for which the set of DLR measures is larger than the set of conformal measures \cite{Sw}. From the book by Aaronson \cite{Aar}, $\sigma$-finite conformal measures are naturally defined, suggesting that it should be possible to consider infinite DLR measures and the generalized Ruelle-Perron-Frobenius Theorem \cite{Sa1} give us the existence of these measures. So far, in all statements in the literature about DLR measures, the authors consider probability measures. On the other hand, the analogous object to an infinite DLR measure in quantum statistical mechanics is the notion of KMS weight (instead of KMS state) already appears more often in the literature of mathematical physics \cite{Chris, Tak, Tho1, Tho2, Tho3}.

We define the $\sigma$-finite but infinite DLR measures and study the relation with $\sigma$-finite conformal measures. Besides, we investigate the equivalence between these two notions. We show that every $\sigma$-finite conformal measure is a DLR measure, and we characterize when the converse is true for $\sigma$-invariant DLR measures.

In the case of renewal shifts \cite{Io, Sa3}, which are examples of countable Markov shifts that do not satisfy the BIP property, O. Sarig showed that the class of weakly H\"{o}lder continuous potentials $\{\beta\phi\}_{\beta>0}$ has a kind of  ``\textit{good-behavior}" (unique critical point) respect to the phase transition in terms of the recurrence mode, i.e., there exists a $\beta_c>0$ (possibly infinite) for which the potential $\beta\phi$ is positively recurrent for $\beta<\beta_c$ (there exists a conservative conformal measure associated to $\beta\phi$), and transient for $\beta>\beta_c$ (does not exist such conservative conformal measure). He also constructed an example of a topologically mixing Markov shift and a potential with an infinite number of critical points that separate intervals where the potential is recurrent and transient alternately. The uniqueness of the critical point, which is a usual property of ferromagnetic systems with pair interactions in statistical mechanics \cite{FV} (see \cite{FS} for ferromagnetic systems with more complicated interactions where the uniqueness is no longer true), also appears in many models already considered in thermodynamic formalism in the ergodic theory literature, see the references in \cite{Sa3}. 

Since the conservative conformal measures associated to $\beta\phi$ can be finite or infinite,  we address the problem of the existence of \textit{volume-type phase transitions} on countable Markov shifts. For the renewal shift we have a expression to the critical inverse temperature $\tilde{\beta}_c$ such that: $\tilde{\beta}_c\le \beta_c$, the eigenmeasures associated to the potential $\beta\phi$ are finite for $\beta<\tilde{\beta}_c$ and, all $\sigma$-finite eigenmeasures associated to the potential $\beta\phi$ when $\tilde{\beta}_c<\beta<\beta_c$ are infinite . Moreover, when $\Var_1\phi<\infty$, there is no volume-type phase transition, meaning that all eigenmeasures, when they exist, are finite for every $\beta > 0$. 

This kind of phase transition (volume-type) is not detected by points where the pressure is not differentiable. The lack of a connection between phase transitions and critical points for the pressure is not new. Even for the most famous model in statistical mechanics, the bidimensional ferromagnetic Ising model, if we add external fields with decay slow enough, we have the DLR state's uniqueness for every temperature. However, the pressure has a unique critical point (the same point as in the case of zero field), see \cite{BCCP, CV}. Another example is the two-dimensional $ XY $ model, the pressure is nonanalytic, and the model presents a phase transition in the sense of Kosterlitz-Thouless. However, it is known that the model has a unique translation-invariant DLR measure, and it is a conjecture, unsolved for more than four decades, to prove that this measure is the only one for the model. For a recent reference about $XY$ model, see \cite{PelS}, and for a discussion about different notions of phase transitions in statistical mechanics in general, see \cite{vEFS}.

The paper is organized as follows: In Section 2, we introduce some definitions and recall previous results. Section 3 is dedicated to the existence of infinite $\sigma$-finite DLR measure, and we prove that every $\sigma$-finite conformal measure is a DLR measure. In Section 4, we investigate when $\sigma$-finite DLR measures are $\sigma$-finite conformal measures, and the relationship between DLR measures and equilibrium measures. In Section 5, we study the volume-type phase transition on renewal shifts.

The results of this paper are mostly contained in the Ph.D. thesis of the first author \cite{Be}.

\section{Preliminaries and previous results}\label{sec:notation}
\noindent

Let $S:=\mathbb{N}$ be the set of {\it{states}} and $A=\left(A(i,j)\right)_{S\times S}$ a {\it{transition matrix}} of zeroes and ones with no columns or rows which are all zeroes. Let $\mathbb{N}_0 = \bbn\cup \{0\}$,  the {\it{topological Markov shift}} is the set
$$\Sigma_{A}:=\left\{x=(x_0,x_1,x_2,\ldots)\in S^{\mathbb{N}_0}:A(x_i,x_{i+1})=1,\forall i\geq0\right\},$$
equipped with the topology generated by the collection of {\it{cylinders}}
$$[a_0,a_1,\ldots,a_{n-1}]:=\{x\in \Sigma_{A}:x_i=a_i,0\leq i\leq n-1\},$$
where $n\in\bbn$ and $a_i\in S$, $0\leq i\leq n-1$. We denote by $\mathcal{B}$ the Borel $\sigma$-algebra of $\Sigma_A$, that is the smallest $\sigma$-algebra containing the topology generated by the cylinders. An {\it{admissible word}} of length $n$, denoted by $\underline{a}$, is an element of $S^n$ satisfying $[\underline{a}]\neq\emptyset$. The function $\sigma:\Sigma_{A}\to\Sigma_{A}$ defined by $(\sigma x)_i=x_{i+1}$ for every $i\ge 0$ is called the {\it{shift map}}.

The topological Markov shift $\Sigma_{A}$ is {\it{transitive}} if for every $a,b\in S$ there exists $N\in\bbn$ such that $[a]\cap\sigma^{-N}[b]\neq\emptyset$ and it is {\it{topologically mixing}} if for every $a,b\in S$ there exists $N\in\bbn$ such that, for all $n>N$, we have $[a]\cap\sigma^{-n}[b]\neq\emptyset$. We say that $\Sigma_{A}$ satisfies the {\it BIP property} if there exist $N\ge 1$ and $b_1,b_2,\ldots,b_N\in S$ such that, for all $a\in S$, there exist $1\leq i,j\leq N$ such that $A(a,b_i)=A(b_j,a)=1$. We say that $\Sigma_A$ is \emph{row finite} if 
\begin{equation*}
\sum_{b\in S}A(a,b)<\infty \quad \text{for every }a\in S.
\end{equation*}
Note that every row-finite topological Markov shift $\Sigma_A$ is locally compact.

\begin{defi}
The \emph{renewal shift} is the topological Markov shift with the transition matrix $\left(A(i,j)\right)_{S\times S}$ whose entries $A(1,1),A(1,i)$ and $A(i,i-1)$ are equal to $1$ for every $i>1$, and the other entries are equal to $0$.
\end{defi}

Note that the renewal shift is topologically mixing and does not satisfy the BIP property. 
%\bfblue{falar mais sobre renewal, porque \'{e} importante, e referencias}

\begin{figure}[!htb]
\begin{center}
\begin{tikzpicture}[scale=1.5]
\draw   (0,0) -- (7,0);
\node[circle, draw=black, fill=white, inner sep=1pt,minimum size=5pt] (1) at (0,0) {1};
\node[circle, draw=black, fill=white, inner sep=1pt,minimum size=5pt] (2) at (1,0) {2};
\node[circle, draw=black, fill=white, inner sep=1pt,minimum size=5pt] (3) at (2,0) {3};
\node[circle, draw=black, fill=white, inner sep=1pt,minimum size=5pt] (4) at (3,0) {4};
\node[circle, draw=black, fill=white, inner sep=1pt,minimum size=5pt] (5) at (4,0) {5};
\node[circle, draw=black, fill=white, inner sep=1pt,minimum size=5pt] (6) at (5,0) {6};
\node[circle, draw=black, fill=white, inner sep=1pt,minimum size=5pt] (7) at (6,0) {7};
\node (8) at (7,0) {};
\draw[->, >=stealth] (2)  to (1);
\draw[->, >=stealth] (3)  to (2);
\draw[->, >=stealth] (4)  to (3);
\draw[->, >=stealth] (5)  to (4);
\draw[->, >=stealth] (6)  to (5);
\draw[->, >=stealth] (7)  to (6);
\draw[->, >=stealth] (8)  to (7);
\node [minimum size=10pt,inner sep=0pt,outer sep=0pt] {} edge [in=200,out=100,loop, >=stealth] (0);
\draw[->, >=stealth] (1)  to [out=90,in=90, looseness=1] (2);
\draw[->, >=stealth] (1)  to [out=90,in=90, looseness=1] (3);
\draw[->, >=stealth] (1)  to [out=90,in=90, looseness=1] (4);
\draw[->, >=stealth] (1)  to [out=90,in=90, looseness=1] (5);
\draw[->, >=stealth] (1)  to [out=90,in=90, looseness=1] (6);
\draw[->, >=stealth] (1)  to [out=90,in=90, looseness=1] (7);
\end{tikzpicture}
\end{center}
\caption{Renewal shift.}\label{fig:renewal}
\end{figure}
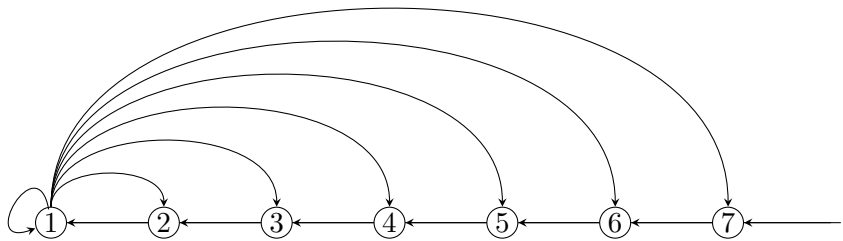

A function $\phi:\Sigma_{A}\to\bbr$ is called a {\it {potential}}. For every $n\geq 1$ and potential $\phi$, the {\it $n$-variation of $\phi$} is given by
\begin{equation*}
\Var_n\phi:=\sup\left\{\lvert\phi(x)-\phi(y)\rvert :~x,y\in\Sigma_{A}, x_i=y_i,0\leq i\leq n-1\right\}.
\end{equation*}

A potential $\phi$ is called {\it{weakly H\"{o}lder continuous}} if there exist $H_{\phi}>0$ and $\theta\in(0,1)$ such that for all $n\geq 2$, $\Var_n\phi\leq H_{\phi}\theta^n$. We say that a potential $\phi$ has {\it summable variation} if $\sum_{n\geq2}\Var_n\phi<\infty$.
 
For every $n\in\bbn$, we define $\phi_n(x):=\sum_{i=0}^{n-1}\phi(\sigma^ix)$ the $n$-th ergodic sum. A potential $\phi$ satisfies the {\it Walters condition} if:
$$
 \sup_{n\geq 1}\Var_{n+k}\phi_n<\infty,\text{ for every } k\geq1,\ \text{ and }\lim_{k\to \infty}\sup_{n\geq 1}\Var_{n+k}\phi_n= 0.
$$
Note that every potential with summable variation satisfies the Walters condition. Here, we allow potentials $\phi$ satisfying $\Var_1\phi=\infty$.

Two potentials $\phi,\varphi:\Sigma_{A}\to\mathbb{R}$ are called {\it cohomologous} via a function $h:\Sigma_A\to\mathbb{R}$ if $\phi=\varphi+h-h\circ\sigma$, and it will be denoted by $\phi\sim\varphi$. A function $f:\Sigma_A\to \mathbb{R}$ is \emph{bounded away from zero} if $\inf_{x\in \Sigma_A} f(x)>0$, and $f$ is \emph{bounded away from infinity} if $\sup_{x\in \Sigma_A} f(x)<\infty$.

For every $n\geq 1$ and $a\in S$, set 
$$
Z_n(\phi,a):=\sum_{\sigma^nx=x}e^{\phi_n(x)}\mathbbm{1}_{[a]}(x)\ \mbox{ and }\ Z_n^{*}(\phi,a):=\sum_{\sigma^nx=x}e^{\phi_n(x)}\mathbbm{1}_{[\phi_a=n]}(x),
$$
where $\phi_a(x)=\mathbbm{1}_{[a]}(x)\inf\{n\geq 1:\sigma^nx\in[a]\}$ (where $\inf \emptyset:=\infty$ and $0.\infty:=0$). The {\it Gurevich pressure} of $\phi$ is defined by
\be\label{Gur}
P_{G}(\phi):=\lim\limits_{n\rightarrow\infty}\frac{1}{n}\log Z_n(\phi,a).
\ee
Due to Sarig \cite{Sa1}, the limit $(\ref{Gur})$ exists and does not depend on $a\in S$ if $\Sigma_{A}$ is topologically mixing and $\phi$ satisfies the Walters condition. 

Denote by $\mathcal{M}^1(\Sigma_{A})$ the set of probability measures on $\Sigma_{A}$ and $\mathcal{M}_{\sigma}^1(\Sigma_{A})$ the set of $\sigma$-invariant probability measures on $\Sigma_{A}$. If $\sup\phi<\infty$, the Gurevich pressure can be expressed by 
\be\label{equi-meas}
P_G(\phi)=\sup\bigg\{h_{\nu}(\sigma)+\int \phi \d\nu:~\nu\in\mathcal{M}_{\sigma}^1\big(\Sigma_{A}\big)\mbox{ s.t. }-\int \phi \d\nu<\infty\bigg \},
\ee
where $h_{\nu}(\sigma)$ is the metric entropy of $\nu$.

A measure $\mu\in\mathcal{M}_{\sigma}^1\big(\Sigma_{A}\big)$ is an {\it equilibrium measure} for $\phi$ if the supremum of (\ref{equi-meas}) is attained for $\mu$, i.e., 
$$P_G(\phi)=h_{\mu}(\sigma)+\int\phi\d\mu.$$
 
 Given two $\sigma$-finite measures $\mu$ and $\nu$ in measurable space $\left(\Omega,\calf\right)$, we denote by $\mu \ll \nu$ if, for every $E\in \mathcal{F}$ such that $\nu(E)=0$, we have $\mu(E)=0$. We denote by $\mu\sim \nu$ if $\mu\ll \nu$ and $\nu \ll \mu$.
 
  A $\sigma$-finite measure $\nu$ in $\calb$ is called \emph{non-singular} if $\nu\circ\sigma^{-1}\sim\nu$. Define $\nu\circledcirc\sigma$ the measure in $\calb$ given by
$$
\nu\circledcirc\sigma(E):=\sum_{a\in S}\nu\left(\sigma\left(E\cap [a]\right)\right).
$$

Note that $\nu\ll\nu\circledcirc\sigma$ when $\nu$ is a non-singular measure. The $\sigma$-finite non-singular measure $\nu$ is called {\it conservative} if every set $W\in\calb$ such that $\{\sigma^{-n}W\}_{n\geq 0}$ is pairwise disjoint mod $\nu$ satisfies $W=\emptyset\mod \nu$. These $\sigma$-finite measures satisfy the Poincar\'{e} Recurrence Theorem, see Theorem $2.1$ in \cite{Sa5}.

\begin{defi}
Let $\Sigma_A$ be a topological Markov shift, $\nu$ a $\sigma$-finite measure in $\calb$ and $\phi:{\Sigma_A}\to\mathbb{R}$ a measurable potential. For a fixed $\lambda>0$, the measure $\nu$ is called \emph{$(\phi,\lambda)$-conformal} if
\begin{equation*}
\frac{\d\nu}{\d\nu\circledcirc\sigma}=\lambda^{-1}e^{\phi},\quad\nu\circledcirc\sigma\mbox{-a.e.}
\end{equation*} 
\end{defi}

%An important tool to study the conformal, the equilibrium measures and the Gurevich pressure, is the 
For a fixed potential $\phi:\Sigma_{A}\to\mathbb{R}$,
the {\it Ruelle operator} $L_{\phi}$
%\cite{Ru} which 
is defined by
\begin{equation}\label{op-Ruelle}
L_{\phi}f(x):=\sum_{\sigma(y)=x}e^{\phi(y)}f(y).
\end{equation} 
When $|S|<\infty$,  the Ruelle operator is well defined for every continuous function $f$. However, for countable Markov shifts, we need conditions to be well-defined. See Theorems \ref{teo-RPF}, \ref{RPFt} and \cite{MaUr}. 

%For finite alphabet $S$, we know that the dual of the Ruelle operator $L^*_{\phi}$ is well-defined by Riesz Lemma, i.e., for every continuous function $f$,
%$$
%\int L_{\phi}f(x)\d\nu(x)=\int f(x)\d L^*_{\phi}\nu(x).
%$$

The following proposition shows that $(\phi,\lambda)$-conformal measures are eigenmeasures of the Ruelle operator.

\begin{prop}\label{conf-auto} Let $\Sigma_{A}$ be a topological Markov shift, $\nu$ a $\sigma$-finite measure in $\calb$, $\phi:{\Sigma_A}\to\mathbb{R}$ a measurable potential and $\lambda>0$. Then $\nu$ is $(\phi,\lambda)$-conformal if, only and if,
\be\label{dual}	
\int L_{\phi}f(x)\d\nu(x)=\lambda\int f(x)\d\nu(x),\quad \mbox{ for each }f\in L^1(\nu).
\ee
\end{prop}
Equation $(\ref{dual})$ will be denoted simply by $L_{\phi}^{*}\nu=\lambda\nu$.

Proposition \ref{prop:transferencia} is a particular result of Proposition 1.4.1 in \cite{Aar} and Proposition 2.3 in \cite{Sa5}.

\begin{prop}\label{prop:transferencia}
	Let $\Sigma_{A}$ be a topological Markov shift and $\nu$ a $\sigma$-finite measure. If  $\nu$ is $\sigma$-invariant, then
	\begin{equation*}
	\sum_{\sigma y=x}\frac{\d \nu}{\d \nu\circledcirc\sigma}(y)=1,\quad \nu\text{-a.e}.
	\end{equation*}
	
\end{prop}

\begin{defi}\label{modes-recu} Let ${\Sigma_A}$ be a topologically mixing Markov shift and $\phi:{\Sigma_A}\to\mathbb{R}$ a potential satisfying the Walters condition such that $P_G(\phi)< \infty$. Fix some $a\in S$. We say that the potential is:
	\begin{itemize}
		\item[i)] \emph{Recurrent} if $\sum_{n\geq1}e^{-nP_G(\phi)}Z_n(\phi,a)=\infty$.
		\item[ii)] \emph{Positive recurrent} if $\phi$ is recurrent and $\sum_{n\geq1}ne^{-nP_G(\phi)}Z_n^{*}(\phi,a)<\infty$.
		\item[iii)] \emph{Null recurrent} if $\phi$ is recurrent and $\sum_{n\geq1}ne^{-nP_G(\phi)}Z_n^{*}(\phi,a)=\infty$.
		\item[iv)] \emph{Transient} if $\sum_{n\geq1}e^{-nP_G(\phi)}Z_n(\phi,a)<\infty$.
	\end{itemize}
\end{defi}

Since $\Sigma_A$ is topologically mixing, all modes of recurrence defined above are independent of $a\in S$. When $|S|<\infty$ we have that any $\phi$ is positive recurrent. The following theorem given by O. Sarig characterizes each mode of recurrence.

\begin{teo}[Generalized Ruelle-Perron-Frobenius Theorem, \cite{Sa1}]\label{teo-RPF} Let ${\Sigma_A}$ be a topologically mixing Markov shift, $\phi:{\Sigma_A}\to\mathbb{R}$ a potential that satisfies the Walters condition and $P_G(\phi)<\infty$. Then:
	\begin{itemize}
		\item[i)] $\phi$ is positive recurrent if, and only if, there exist $\lambda>0$, a positive continuous function $h$, 
		
\hspace{-0.55cm} and a conservative measure $\nu$ which is finite on cylinders, such that $L_{\phi}h=\lambda h$, $L_{\phi}^{*}\nu=\lambda\nu$, 

\hspace{-0.55cm} and $\int h\d\nu=1$. In this case $\lambda=e^{P_G(\phi)}$.
		\item[ii)] $\phi$ is null recurrent if, and only if, there exist $\lambda>0$, a positive continuous function $h$, and
		
\hspace{-0.55cm} a conservative measure $\nu$ which is finite on cylinders, such that $L_{\phi}h=\lambda h$, $L_{\phi}^{*}\nu=\lambda\nu$, and 

\hspace{-0.55cm} $\int h\d\nu=\infty$. In this case $\lambda=e^{P_G(\phi)}$.
		\item[iii)] $\phi$ is transient if, and only if, there is no conservative measure $\nu$ which is finite on cylinders 
		
\hspace{-0.55cm} such that $L_{\phi}^{*}\nu=\lambda\nu$ for some $\lambda>0$.
	\end{itemize}
\end{teo}

The previous theorem says nothing about the finiteness of the measure $\nu$. In general, this could be infinite, as shown in Example \ref{exem1}. But it is known that when $\Sigma_{A}$ satisfies the BIP property and $\Var_1\phi<\infty$ then $\nu$ is finite \cite{Sa4}.

For positively recurrent potential $\phi$, under the conditions of Theorem \ref{teo-RPF}, Sarig \cite{Sa1} showed that $m:=h\d \nu$ is an invariant probability measure, which we call \emph{Ruelle-Perron-Frobenius (RPF) measure}. Moreover, if $h_m(\sigma)<\infty$, then $m$ is the unique equilibrium measure for $\phi$.

Theorem \ref{teo-RPF} guarantees the existence of conservative conformal measures and eigenfunctions for the Ruelle operator in the case of recurrent potentials, but not for transient potentials. V. Cyr \cite{Cyr1,Cyr2} studied transient potentials on topological Markov shifts, showing, for instance, the existence of eigenmeasures of the dual of the Ruelle operator. Moreover, O. Shwartz \cite{Sw} showed the existence of the eigenfunctions for the Ruelle operator in the case of locally compact topological Markov shifts, see Theorem \ref{RPFt}.

Let $\Sigma_A$ be a transitive Markov shift, and $\phi$ a potential with summable variation. Let $\lambda>0$. We say that $\phi$ is \emph{$\lambda$-transient} if $\sum_{n\ge 1} \lambda^{-n}Z_n(\phi,a)<\infty$ for some $a\in S$. Note that the item iv) of Definition \ref{modes-recu} is a particular case when $\lambda=e^{P_G(\phi)}$.

\begin{teo}[\cite{Cyr2,Sw}]\label{RPFt}
Let $\Sigma_A$ be a transitive and locally compact topological Markov shift, and $\phi:\Sigma_A\to \mathbb{R}$ a $\lambda$-transient potential with summable variation. There exists a $\sigma$-finite measure $\nu$ in $\Sigma_A$ such that~$\nu$ is positive, finite in each cylinder, and $L^*_{\phi}\nu=\lambda \nu$. Moreover, there exists a continuous function $h:\Sigma_A\to (0,+\infty)$ such that $L_{\phi}h=\lambda h$ and $h\d \nu$ is an invariant finite measure.
\end{teo}

For each $n\in\bbn$, let $\sigma^{-n}\calb$ denote the smallest $\sigma$-algebra in which all the coordinate functions $\pi_k:\Sigma_{A}\to S$ given by $\pi_k(x)=x_k$, with $k\geq n$, are measurable. Thus, we have the following family of $\sigma$-algebras
$$
\calb\supset \sigma^{-1}\calb \supset \sigma^{-2}\calb\supset \ldots \supset \sigma^{-n}\calb\supset\ldots
$$

\begin{defi}\label{def:dlrprob}
Let $\Sigma_{A}$ be a topological Markov shift, $\nu$ a probability measure in $\calb$ and $\phi:\Sigma_{A}\to \bbr$ a measurable potential. We say that $\nu$ is \emph{$\phi$-DLR} if for every $n\ge 1$ and for every cylinder $[\underline{a}]$ of length $n$, we have 
\begin{equation}\label{dlrprob}
	\mathbb{E}_{\nu}\big(\mathbbm{1}_{[\underline{a}]}|\sigma^{-n}\mathcal{B}\big)(x)=\frac{e^{ \phi_n(\underline{a}\sigma^nx)}\mathbbm{1}_{\{\underline{a}\sigma^nx\in \Sigma_{A} \}}}{\displaystyle\sum_{\sigma^n y=\sigma^n x}e^ {\phi_n(y)}},\qquad \nu\text{-a.e.}
\end{equation}
\end{defi}

Equations (\ref{dlrprob}) are called DLR \emph{equations}, see also \cite{Do1, Do2, Do3, FV, Geo, LaRu, RaSe}. The next result, by Sarig \cite{Sa5}, gives general conditions for Markov shift and potentials such that any conformal probability measure is a DLR measure. The reciprocal is not always true, see example \ref{exem4}.

\begin{teo}[\cite{Sa5}]\label{Sa5}
Let $\Sigma_{A}$ be a topological Markov shift and $\phi:\Sigma_{A}\to\bbr$ a measurable potential. Then any non-singular  $(\lambda,\phi)$-conformal probability measure $\nu$ is a \emph{$\phi$-DLR} measure.
\end{teo}

\section{Infinite DLR measures}\label{sec:infinite}

In order to define a $\sigma$-finite DLR measure $\nu$ with $\nu(\Sigma_A)=\infty$, we need that the family of the conditional expectations $\left\{\mathbb{E}_{\nu}[\cdot|\sigma^{-n}\mathcal{B}]\right\}_{n\geq 1}$ should be well-defined, i.e., for each $n\ge 1$, $\nu$ is $\sigma$-finite in the sub-$\sigma$-algebra $\sigma^{-n}\mathcal{B}$.

Let us define  $\mathcal{M}(\Sigma_{A})$ be the set of $\sigma$-finite measures (not necessarily probability measures) on $\Sigma_A$, and $\mathcal{M}_{\sigma}(\Sigma_{A})$ be the set of $\sigma$-invariant $\sigma$-finite measures on $\Sigma_{A}$. We say that a sub-$\sigma$-algebra $\mathcal{F}$ is \emph{compatible} with the measure $\nu\in \mathcal{M}(\Sigma_{A})$ if $\nu$ is $\sigma$-finite in $\mathcal{F}$.

\begin{defi}\label{def:dlrinf}
Let $\Sigma_{A}$ be a topological Markov shift, $\nu$ a $\sigma$-finite measure in $\calb$ and $\phi:\Sigma_{A}\to \bbr$ a measurable potential. We say that $\nu$ is \emph{$\phi$-DLR} if, for every $n\ge 1$,
\begin{itemize}
\item[i)] the sub-$\s$-algebra $\sigma^{-n}\mathcal{B}$ is compatible with the measure $\nu$,
\item[ii)] for every cylinder $[\underline{a}]$ of length $n$, we have 
\begin{equation}\label{dlrinf}
	\mathbb{E}_{\nu}\big(\mathbbm{1}_{[\underline{a}]}|\sigma^{-n}\mathcal{B}\big)(x)=\frac{e^{ \phi_n(\underline{a}\sigma^nx)}\mathbbm{1}_{\{\underline{a}\sigma^nx\in \Sigma_{A} \}}}{\displaystyle\sum_{\sigma^n y=\sigma^n x}e^ {\phi_n(y)}},\qquad \nu\text{-a.e.}
\end{equation}
\end{itemize} 
\end{defi}

Note that when $\nu$ is a probability measure, Definition \ref{def:dlrinf} coincides with Definition \ref{def:dlrprob}. The following proposition shows a class of measures satisfying item $i)$ of the previous definition.

\begin{prop}\label{fo1} 
Consider $\Sigma_{A}$ be a topological Markov shift, $\phi:\Sigma_{A}\to \bbr$ a measurable potential and $\nu$ a $(\phi,\lambda)$-conformal, for some $\lambda>0$, such that $\nu\left([a]\right)<\infty$ for every $a\in S$.  If $\|L_{\phi}\mathbbm{1}\|_{\infty}<\infty$, then $\nu\left(\pi^{-1}_n\{a\}\right)<\infty$ for every $n\ge 1$ and $a\in S$. In particular, $\s^{-n}\calb$ is compatible with the measure~$\nu$ for each $n\geq 1$.
\end{prop}  

\begin{proof} 
For a fixed $n\ge 1$ and $a\in S$, let $(w_0,w_1,\ldots, w_{n-1})$ be a word of length $n$ such that $A(w_{n-1},a)=1$. Thus
\begin{eqnarray}\label{ojo}
	\lambda^n \nu\big([w_0,w_1,\ldots,w_{n-1},a]\big)=\int_{[a]}e^{\phi_n(w_0,w_1,\ldots,w_{n-1},x)}\d\nu(x).
\end{eqnarray}
Note that
	\begin{equation*}
 \pi^{-1}_n\{a\}=\bigcup_{\begin{subarray}{c}
		[w_0,w_1,\ldots,w_{n-1}]\neq\emptyset \\A(w_{n-1},a)=1\end{subarray}}[w_0,w_1,\ldots,w_{n-1},a].
	\end{equation*}
Take the sum in (\ref{ojo}) over all cylinder $[w_0,w_1,\ldots, w_{n-1}]$ such that $A(w_{n-1},a)=1$. By Monotone Convergence Theorem,
	\begin{eqnarray*}
		\nu\big(\pi_{n}^{-1}\{a\}\big)&=&\lambda^{-n}\sum_{\begin{subarray}{c}
				[w_0,w_1,\ldots,w_{n-1}]\neq\emptyset \\A(w_{n-1},a)=1\end{subarray}}\int_{[a]}e^{\phi_n(w_0,w_1,\ldots,w_{n-1},x)}\d\nu(x)\\
		&=&\lambda^{-n}\int_{[a]}L_{\phi}^n\mathbbm{1}d\nu\\
		&\leq&\lambda^{-n}\nu\left([a]\right)\|L_{\phi}\mathbbm{1}\|_{\infty}^{n},
	\end{eqnarray*}
	which is finite since $\|L_{\phi}\mathbbm{1}\|_{\infty}<\infty$.
\end{proof}

In the next example, we show that the condition $\|L_{\phi}\mathbbm{1}\|_{\infty}<\infty$ does not imply that the conformal measure is finite.
%A seguir mostraremos que para potenciais $\phi:\Sigma_{A}\rightarrow\mathbb{R}$ tal que $\|L_{\phi}\mathbbm{1}\|_{\infty}<\infty$ existem automedidas $\sigma-$finitas tal que $\nu\big( \Sigma_{A} \big)=+\infty$.

\begin{exem}\label{exem1} Consider the renewal shift and a potential $\phi:\Sigma_{A}\to \mathbb{R}$ given by  $\phi(x)={x_0-x_1}$. Note that $\phi$ satisfies the Walters condition, $\|L_{\phi}\mathbbm{1}\|_{\infty}<\infty$, and $P_G(\phi)=\log2$. Let $\lambda:=e^{P_G(\phi)}$. The expression $Z_n(\phi,1)=2^{n-1}$ implies that $\phi$ is recurrent. By Generalized Ruelle-Perron-Frobenius Theorem, there exists a positive measure $\nu$ finite in cylinders such that 
\be\label{exemplodual}
\int_{\Sigma_{A}} L_{\phi}f\d\nu=\lambda\int_{\Sigma_{A}} f\d\nu
\ee	
for every $f\in L^1(\nu)$. For each $a\geq 2$, consider the function $f = \mathbbm{1}_{[a]}$. Substituting in Equation (\ref{exemplodual}), we have $\nu\left([a]\right)=\frac{e}{2}\nu\left([a-1]\right)$. Therefore $\nu\left(\Sigma_{A}\right)=+\infty$.
\end{exem}

\begin{note}\label{Remark1} O. Sarig \cite{Sa3} showed that a weakly H\"older continuous potential $\phi$ defined in the renewal shift has good behavior with respect to the phase transition in the recurrence modes, that is, there exists $\beta_c\in(0,\infty]$ such that $\beta \phi$ is positive recurrent for $\beta<\beta_c$, and transient for $\beta>\beta_c$. For every $\beta>0$, consider the family of potentials $\{\beta\phi\}_{\beta>0}$ where $\phi$ is the potential from Example \ref{exem1}, a direct calculation shows that $\beta\phi$ is positive recurrent for all $\beta>0$. Note that $\nu_{\beta}(\Sigma_{A})$ is finite for $\beta<\log2$ and infinite for $\beta>\log2$, where $\nu_{\beta}$ be the eigenmeasure associate to the potential $\beta\phi$. Then, there is no phase transition in the recurrence mode in this case, but there is a phase transition in the sense of the conformal measure's finiteness. We will prove that the volume-type phase transition on renewal shifts for weakly H\"older continous potentials has also a good behavior in Section \ref{sec:phase-transition}.
\end{note}

The following corollary is an extension of a result proved by V. Cyr in \cite{Cyr1} for the positive recurrent case. Now, since we have a definition of an infinite volume DLR measure, we are able to deal with the null recurrent case when $h\d\nu$ is infinite.

\begin{coro}\label{coro-Cyr}
Let $\Sigma_{A}$ be a topologically mixing Markov shift, $\phi:\Sigma_{A}\to \bbr$ a potential sa-tisfying the Walters condition and $P_G(\phi)<\infty$. If $\phi$ is a recurrent potential then $h\d\nu$ is a  $\left(\phi+ \log{h}-\log{h\circ\sigma}-P_G(\phi)\right)$-\emph{DLR} measure, where $h$ is a positive continuous function and $\nu$ is a measure such that $L_{\phi}h=e^{P_G(\phi)}h$ and $L^{*}_{\phi}\nu=e^{P_G(\phi)}\nu$.
\end{coro}

\begin{proof}  Note that the hypotheses of the corollary imply that $h \d\nu\left(\pi^{-1}_n\{a\}\right)<\infty$ for every $n \geq 1$ and then the sub-$\s$-algebra $\sigma^{-n}\mathcal{B}$ is compatible with the measure $h \d\nu$. The rest of the proof follows as in \cite{Cyr1}. 
\end{proof}

We remember that we can define the conditional expectation for $\sigma$-finite measures, in particular, we have the Martingale Convergence Theorem for $\sigma$-finite measures.

\begin{teo}\label{TMG}
Let $(\Omega,\mathcal{B},\mu)$ be a measure space, and $\{\mathcal{F}_i\}_{i\ge 1}$ a family of sub-$\sigma$-algebras of $\mathcal{B}$, each one compatible with the measure $\mu$, satisfying $\mathcal{F}_i\subseteq \mathcal{F}_{i+1}$ for every $i\ge 1$. Consider the $\sigma$-algebra $\mathcal{F}:=\sigma\left(\cup_{n\geq 1}\mathcal{F}_n\right)$. If $f\in L^1\left(\Omega,\mathcal{B},\mu\right)$ and $\mathcal{F}$ is compatible with the measure $\mu$, then
	\begin{equation*}
	\lim_{n\to \infty}\mathbb{E}[f\mid\mathcal{F}_n]=\mathbb{E}[f\mid\mathcal{F}]
	\end{equation*}
$\mu$-a.e. and in $L^1(\mu)$.
\end{teo}	

\begin{proof}
The proof is adapted as in Theorem 5.5 of \cite{EiWa}, using Approximation Theorem (see Theorem 4.4 in \cite{Ta}).
\end{proof}

Proposition \ref{resu2} states a characterization of the DLR measures, which is analogous to the pro\-ba\-bility measure's case given by Sarig in \cite{Sa5}. The proof is an adaptation of the Proposition $2.1$ in \cite{Sa5} using Theorem~\ref{TMG}.

\begin{prop}\label{resu2}
Let $\Sigma_{A}$ be a topological Markov shift, $\phi:\Sigma_{A}\to \bbr$ a measurable potential and $\nu$ a measure such that $\nu\left(\pi^{-1}_n\{s\}\right)<\infty$ for every $n\geq0$ and $s\in S$. Then $\nu$ is $\phi$-\emph{DLR} measure if, and only if, for every pair of cylinders $[\underline{a}]=[a_0,a_1,\ldots,a_{n-1}]$ and $[\underline{b}]=[b_0,b_1,\ldots,b_{n-1}]$ of length $n\in\bbn$ such that $a_{n-1}=b_{n-1}$ and $\nu\left([\underline{a}]\right)>0$, the map $v_{\underline{a},\underline{b}}:[\underline{a}]\rightarrow[\underline{b}]$ given by $v_{\underline{a},\underline{b}}(\underline{a}x_n^{\infty})=(\underline{b}x_n^{\infty})$ satisfies
	\begin{equation}\label{vab}
	\frac{\d\nu\circ v_{\underline{a},\underline{b}}}{\d\nu}=e^{\phi_n(\underline{b}x_n^{\infty})-\phi_n(\underline{a}x_n^{\infty})},\quad\nu\text{-a.e.}\mbox{ in }[\underline{a}].
	\end{equation}	
\end{prop}

%\bfblue{\textit{Remark 2}: The measure $\nu$ needs the condition $\nu\left([s]\right)<\infty$ for every $s\in S$ only to prove the reciprocal.}

Theorem \ref{con-dlr} below states that every $(\phi,\lambda)$-conformal measure with $\lambda>0$ such that the event $\pi^{-1}_n\{a\}$ has finite mass for all $n\ge 0$ and $a\in S$ is a $\phi$-DLR measure. In particular, for recurrent potentials $\phi$ with $\|L_{\phi}\mathbbm{1}\|_{\infty}<\infty$, by Proposition \ref{fo1}, all conformal measures from Generalized Ruelle-Perron-Frobenius Theorem are $\phi$-DLR measures. Moreover, if $\Sigma_A$ is a transitive shift with row finite, and $\phi:\Sigma_A\to\mathbb{R}$ is a $\lambda$-transient potential for some $\lambda>0$ with summable variation satisfying $\|L_{\phi}\mathbbm{1}\|_{\infty}<\infty$, all conformal measures from Theorem \ref{RPFt} are $\phi$-DLR measures.

\begin{teo}\label{con-dlr}
Let $\Sigma_{A}$ be a topological Markov shift, $\phi:\Sigma_{A}\to \bbr$ a measurable potential, and let~$\nu$ be a measure satisfying $\nu\left(\pi^{-1}_n\{a\}\right)<\infty$ for every $n\ge 0$ and $a\in S$. If $\nu$ is a non-singular and $(\phi,\lambda)$-conformal measure for some $\lambda>0$, then $\nu$ is a $\phi$-\emph{DLR} measure.
\end{teo}

\begin{proof} 
The proof is an adaptation of Proposition $2.2$ in \cite{Sa5}, using Proposition \ref{resu2}.
\end{proof}

\section{When the DLR measures are the conformal measures}\label{sec:DLRimpliesconformal}

In this section we investigate when a DLR measure is a conformal measure. In addition to that we will study the uniqueness of the DLR measure when the Markov shift satisfies the BIP property, as well as the connection with the equilibrium measures. 
%The following lemma and proposition will be important for the proof of Theorem \ref{por1G}.

\begin{lema}\label{composG}
Let $\Sigma_{A}$ be a topological Markov shift, $\phi:\Sigma_{A}\to \bbr$ a measurable potential, and $\mu$ a non-singular $\phi$-\emph{DLR} measure. Consider $a,c\in S$ such that $A(a,c)=1$. If $\mu\left([ac]\right)>0$, then
	\begin{enumerate}
		\item[i)]  For every $b\in S$ such that $A(b,c)=1$, we have $\mu\circ v_{ac,bc}\sim \mu$ in $\mathcal{B}\cap[ac]$.
		\item[ii)]  $\mu\circ\sigma\big|_{[ac]}\sim \mu$ in $\mathcal{B}\cap[ac].$
		\item[iii)]  For every $b\in S$ such that $A(b,c)=1$, we have $\mu\circ I_b\sim \mu$ in $\mathcal{B}\cap[c]$.
	\end{enumerate}
Otherwise, if $\mu\left([ac]\right)=0$, then  $\mu\left([c]\right)=0$.
\end{lema}

\begin{proof}
Item i) is straigthforward from Proposition \ref{resu2}. For item ii), since $\mu$ is non-singular and $\mu(E)\leq \mu(\sigma^{-1}\sigma|_{[ac]} E)$ for every $E\in\mathcal{B}\cap[ac]$, we conclude $\mu\ll\mu\circ\sigma\big|_{[ac]}$ in $\mathcal{B}\cap[ac]$.

Let $E_a\in\mathcal{B}\cap[ac]$ satisfying $\mu\big(E_a\big)=0$. For each $b\in S$ with $A(b,c)=1$, consider $E_{b}={v}_{ac,bc}(E_a)$. Note that $E_b\in\mathcal{B}\cap[bc]$. Thus, by Proposition \ref{resu2}, we have $\mu\big(E_b\big)=0$. Since
\begin{equation*}
\mu\left(\sigma^{-1}\sigma\big|_{[ac]} E_a\right)=\sum_{A(b,c)=1}\mu\left(E_b\right)=0,
\end{equation*}
we obtain $\mu\big(\sigma\big|_{[ac]} E_a\big)=0$, and thus $\mu\circ\sigma\big|_{[ac]}\ll \mu$ in $\mathcal{B}\cap[ac]$.

To show item iii), since the map $I_b:[c]\to [bc]$ is a homeomorphism, for every $E\in\mathcal{B}\cap[c]$, we have $I_b(E)\in\mathcal{B}\cap[bc]$.  
Let $E\in \mathcal{B}\cap[c]$ such that $\mu(E)=0$. Since $I_b(E)\subset\sigma^{-1}E$ and $\mu$ is non-singular, we have $\mu\circ I_b(E)=0$. Now, let $E\in\mathcal{B}\cap[c]$ such that $\mu\circ I_b(E)=0$. By item ii), we have  $\mu(E)=\mu(\sigma\big|_{[bc]}(I_b(E)))=0$. Therefore $\mu\circ I_b\sim \mu$ in $\mathcal{B}\cap[c]$.	 

Now, assume $\mu\left([ac]\right)=0$. For every $b\in S$ such that $A(b,c)=1$, we have $v_{ac,bc}[ac]=[bc]$. By Proposition \ref{resu2}, we conclude $\mu\left([bc]\right)=0$. Thus, 
	\begin{equation*}
	m\left(\sigma^{-1}\sigma\vert_{[\omega c]}[\omega c]\right)=\sum_{A(b,c)=1}m\left([bc]\right)=0.
	\end{equation*}
	Since $m$ is non-singular, then $m\left([c]\right)=m\left(\sigma\vert_{[ac]}[ac]\right)=0$.
\end{proof}

\begin{prop}\label{Rus1G} 
Let $\Sigma_{A}$ be a topological Markov shift, $\phi:\Sigma_{A}\to \bbr$ a measurable potential, and $\mu$ be a non-singular $\phi$-\emph{DLR} measure. For every $a,b,c\in S$ such that $A(a,c)=A(b,c)=1$ satisfying $\mu\left([ac]\right)>0$, we have
\be\label{resultado1}
	\frac{\frac{\d\mu}{\d\mu\circledcirc\sigma}(ax)}{e^{\phi(ax)}}=\frac{\frac{\d\mu}{\d\mu\circledcirc\sigma}(bx)}{e^{\phi(bx)}},\quad \mu\text{-a.e.} \mbox{ in }  [c].
\ee
\end{prop}

\begin{proof}
Let $a,b\in S$ such that $A(a,c)=A(b,c)=1$ and $\mu\left([ac]\right)>0$. The map ${v}_{ac,bc}:[ac]\mapsto [bc]$ can be written as ${v}_{ac,bc}=I_b\circ\sigma\big|_{[ac]}$, where $I_b(x)=bx$. 
By item i) and ii) of Lemma \ref{composG},  for every $y\in [ac]$, we have
	\begin{equation*}
	\frac{\d\mu\circ {v}_{ac,bc}}{\d\mu}(y) = \frac{\d\mu\circ I_{b}\circ\sigma\big|_{[ac]}}{\d\mu\circ\sigma\big|_{[ac]}}(y)\cdot \frac{\d\mu\circ\sigma\big|_{[ac]}}{\d\mu}(y),\quad \mu\text{-a.e.}\mbox{ in } [ac],
	\end{equation*}
i.e., the measurable set
	\begin{equation*}
	E=\bigg\{y\in[ac]: \frac{d\mu\circ {v}_{ac,bc}}{d\mu}(y)\neq\frac{d\mu\circ I_{b}\circ\sigma\big|_{[ac]}}{d\mu\circ\sigma\big|_{[ac]}}(y)\cdot \frac{d\mu\circ\sigma\big|_{[ac]}}{d\mu}(y) \bigg\}
	\end{equation*}
	satisfies $\mu(E)=0$. Thus, by item ii) of Lemma \ref{composG}, we have $\mu(\sigma\big|_{[ac]} E)=0$. This implies that the set
	\begin{equation*}
	\sigma|_{[ac]}(E)=\left\{x\in [c]: ~\frac{\d\mu\circ {v}_{ac,bc}}{\d\mu}(ax)\neq\frac{\d\mu\circ I_{b}\circ\sigma\big|_{[ac]}}{\d\mu\circ\sigma\big|_{[ac]}}\big(ax\big)\cdot \frac{\d\mu\circ\sigma\big|_{[ac]}}{\d\mu}\big(ax\big) \right\}
	\end{equation*}
has zero measure. Thus
	\begin{equation}\label{ivi1G}
	\frac{\d\mu\circ {v}_{ac,bc}}{\d\mu}(ax)=\frac{\d\mu\circ I_{b}\circ\sigma\big|_{[ac]}}{\d\mu\circ\sigma\big|_{[ac]}}(ax)\cdot \frac{\d\mu\circ\sigma\big|_{[ac]}}{\d\mu}(ax),\quad \mu\text{-a.e.}\mbox{ in } [c].
	\end{equation}
	By item iii) of Lemma \ref{composG}, we have
	\begin{equation*}
	\frac{\d\mu\circ I_b\circ\sigma\big|_{[ac]}}{\d\mu\circ\sigma\big|_{[ac]}}(y)=\frac{\d\mu\circ I_b}{\d\mu}\circ\sigma\big|_{[ac]}(y),\quad \mu\circ\sigma\big|_{[ac]}\text{-a.e.}\mbox{ in } [ac].
	\end{equation*}
	By the same argument to derive Equation (\ref{ivi1G}), we obtain
	\begin{equation}\label{ivi2G}
	\frac{\d\mu\circ I_b\circ\sigma\big|_{[ac]}}{\d\mu\circ\sigma\big|_{[ac]}}(ax)=\frac{\d\mu\circ I_b}{\d\mu}(x),\quad \mu\text{-a.e.}\mbox{ in } [c].
	\end{equation}
Replacing $(\ref{ivi2G})$ in $(\ref{ivi1G})$,
	\begin{equation}\label{ivi3G}
	\frac{\d\mu\circ {v}_{ac,bc}}{\d\mu}(ax)=\frac{\d\mu\circ I_b}{\d\mu}(x).\frac{\d\mu\circledcirc\sigma\big|_{[ac]}}{\d\mu}(ax),\quad \mu\text{-a.e.}\mbox{ in } [c].
	\end{equation}
	Since $\sigma\big|_{[bc]}\circ I_b=\id$, where $\id$ is the identity function, and  $\mu\circ\sigma\big|_{[bc]}=\mu\circledcirc\sigma\big|_{[bc]}$ in $\mathcal{B}\cap[bc]$,
\be\label{ivi40G}
\frac{\d\mu\circ I_{b}}{\d\mu}(x)=\frac{d\mu\circ I_{b}}{d\mu\circ\sigma\big|_{[bc]}\circ I_{b}}(x)=\frac{\d\mu}{\d\mu\circledcirc\sigma\big|_{[bc]}}(bx),\quad \mu\text{-a.e.}\mbox{ in } [c].
\ee
		By item ii) of Lemma \ref{composG},
			\begin{equation}\label{ivia4G}
		\frac{\d\mu\circledcirc\sigma\big|_{[ac]}}{\d\mu}(ax)=
		\left(\frac{\d\mu}{\d\mu\circledcirc\sigma\big|_{[ac]}}(ax)\right)^{-1},\quad \mu\text{-a.e.}\mbox{ in } [c].
		\end{equation}
			%	\left(\frac{\d\mu}{\d\mu\circledcirc\sigma\big|_{[ac]}}\right)^{-1}(ax)=
	Replacing Equation $(\ref{ivi40G})$ and $(\ref{ivia4G})$ in Equation $(\ref{ivi3G})$, 
	\begin{equation}\label{e1G}
	\frac{\d\mu\circ {v}_{ac,bc}}{\d\mu}(ax)=\frac{\frac{\d\mu}{\d\mu\circledcirc\sigma}(bx)}{\frac{\d\mu}{\d\mu\circledcirc\sigma}({a}x)},\quad \mu\text{-a.e.}\mbox{ in } [c].
	\end{equation}
	By Proposition \ref{resu2} and by the same argument to derive Equation (\ref{ivi1G}),
		\begin{equation}\label{e2G}
	\frac{\d\mu\circ {v}_{ac,bc}}{\d\mu}(ax)=\frac{e^{\phi(bx)}}{e^{\phi(ax)}},\quad \mu\text{-a.e.}\mbox{ in } [c].
	\end{equation}
	Finally, from $(\ref{e1G})$ and $(\ref{e2G})$, we conclude $(\ref{resultado1})$.
\end{proof}

The following result gives information on when a DLR measure is a conformal measure.

\begin{teo}\label{por1G}
Let $\Sigma_{A}$ be a topological Markov shift, $\phi:\Sigma_{A}\to \bbr$ a measurable potential, and $m$ a $\sigma$-finite $\phi$-\emph{DLR} measure. Suppose that $m$ is a $\sigma$-invariant measure. Then $L_{\phi}\mathbbm{1}=\lambda$ $m$-a.e. if, and only if, $m$ is a $(\phi,\lambda)$-conformal measure.
\end{teo} 

\begin{proof}
Let $c\in S$ such that $m([c])>0$. There exists $a\in S$ such that $A(a,c)=1$ and $m\left([ac]\right)>0$. By Proposition \ref{prop:transferencia}, for every $x\in [c]$, we have
	\begin{equation}\label{cuadradoG}
	1=\sum_{\substack{b\in S \\ A(b,c)=1}}\frac{\d m}{\d m\circledcirc\sigma}(bx),\quad m\text{-a.e.}\mbox{ in } [c].
	\end{equation}	
By Proposition \ref{Rus1G}, let $b\in S$ such that $A(b,c)=1$, then
	\begin{equation*} %\label{iG}
	\frac{\d m}{\d m\circledcirc\sigma}(bx)=e^{\phi(bx)}\frac{\frac{\d m}{\d m\circledcirc\sigma}(ax)}{e^{\phi(ax)}},\quad m\text{-a.e.}\mbox{ in } [c].
	\end{equation*}
Summing over all $b\in S$ such that $A(b,c)=1$ and using Equation (\ref{cuadradoG}),
\be\label{iiG}
1=\frac{\frac{\d m}{\d m\circledcirc\sigma}(ax)}{e^{\phi(ax)}}\sum_{\substack{b\in S \\ A(b,c)=1}}e^{\phi(bx)},\quad m\text{-a.e. }\mbox{ in } [c].
\ee
Consider $L_{\phi}\mathbbm{1}=\lambda$. By Equation (\ref{iiG}), we conclude
	\begin{equation*} %\label{conpor1G}
	\frac{\d m}{\d m\circledcirc \sigma}(ax)=\lambda^{-1}e^{\phi(ax)},\quad m\text{-a.e.}\mbox{ in } [c].
	\end{equation*}
Thus, by Lemma \ref{composG}, item ii), for every $F \in \calb \cap [ac]$,
	\begin{equation}\label{inte}
	m(F)=\lambda^{-1}\int_Fe^{\phi(x)}\d m\circledcirc\sigma.
	\end{equation}
	
Let $\mathcal{W}_2(c):=\sigma^{-1}[c]$. Note that every $B\in \calb$ can be written as
	\begin{equation*}
	B=\displaystyle\bigsqcup_{c\in S}\displaystyle\bigsqcup_{\omega\in\mathcal{W}_2(c)}B\cap[\omega].
	\end{equation*}
	By Lemma \ref{composG} and Equation (\ref{inte}),
		\begin{eqnarray*}
	m(B)&=&\sum_{c\in S}\sum_{\begin{subarray}{c}\omega\in\mathcal{W}_2(c)\\m([\omega])>0\end{subarray}}m\left(B\cap[\omega]\right)\\
        &=&\sum_{c\in S}\sum_{\begin{subarray}{c}\omega\in\mathcal{W}_2(c)\\m([\omega])>0\end{subarray}} \lambda^{-1}\int_{B\cap[\omega]}e^{\phi(x)}\d m\circledcirc\sigma\\
         &=& \lambda^{-1}\sum_{c\in S}\sum_{\omega\in\mathcal{W}_2(c)}\int_{B\cap[\omega]}e^{\phi(x)}\d m\circledcirc\sigma\\
         &=& \lambda^{-1}\int_{B}e^{\phi(x)}\d m\circledcirc\sigma.
	\end{eqnarray*}
Concluding that $m$ is a $(\phi,\lambda)$-conformal measure.
	
	Let us suppose that $m$ is a $(\phi,\lambda)$-conformal measure. From Equation $(\ref{iiG})$ we have that, for every $c\in S$ with $m\left([c]\right)>0,$
		\begin{equation*}
		\sum_{\sigma y=x}e^{\phi(y)}=\lambda,\quad m\mbox{-a.e. in }[c].
	\end{equation*}
	Therefore $L_{\phi}\mathbbm{1}(x)=\lambda$, $m$-a.e.
\end{proof}

The next example satisfies all conditions of Theorem \ref{por1G}.

%The following example we have a Markov shift, a potential and a infinite DLR measure satisfying the hypotheses the previous result.

\begin{exem}\label{exem3} Consider the topological Markov shift $\Sigma_{A}$ defined by the graph of Figure \ref{fig:linha}, and let $\phi:\Sigma_{A}\to \mathbb{R}$ given by $\phi\equiv 0$. Note that $\phi$ normalizes the Ruelle operator, i.e., $L_{\phi} \mathbbm{1}=1$.

\begin{figure}[!htb]
\begin{center}
		\begin{tikzpicture}[scale=1.5]
		\node[circle, draw=black, fill=white, inner sep=1pt,minimum size=5pt] (0) at (0,0) {0};
		\node[circle, draw=black, fill=white, inner sep=1pt,minimum size=5pt] (1) at (1,0) {1};
		\node[circle, draw=black, fill=white, inner sep=1pt,minimum size=5pt] (2) at (-1,0) {2};
		\node[circle, draw=black, fill=white, inner sep=1pt,minimum size=5pt] (3) at (2,0) {3};
		\node[circle, draw=black, fill=white, inner sep=1pt,minimum size=5pt] (4) at (-2,0) {4};
		\node[circle, draw=black, fill=white, inner sep=1pt,minimum size=5pt] (5) at (3,0) {5};
		\node[circle, draw=black, fill=white, inner sep=1pt,minimum size=5pt] (6) at (-3,0) {6};
		\node (i) at (-4,0) {$\ldots$};
		\node (f) at (4,0) {$\ldots$};
		\draw[->, >=stealth] (i)  to (6);
		\draw[->, >=stealth] (6)  to (4);
		\draw[->, >=stealth] (4)  to (2);
		\draw[->, >=stealth] (2)  to (0);
		\draw[->, >=stealth] (0)  to (1);
		\draw[->, >=stealth] (1)  to (3);
		\draw[->, >=stealth] (3)  to (5);
		\draw[->, >=stealth] (5)  to (f);
		\end{tikzpicture}
	\end{center}
	\caption{}\label{fig:linha}
\end{figure}
	
    For $i\ge 0$, consider the sequence	
	$$  x_{i}= \left\{
	\begin{array}{ll}
	(i,i+2,i+4\ldots),    & \mbox{ if }i \mbox{ is odd};\\
	(0,1,3,5,\ldots),      & \mbox{ if }i=0;\\
	(i,i-2,\ldots,i-2k+2,0,1,\ldots),                &\mbox{ if } i=2k,~k\ge 1;
	\end{array}
	\right.
	$$
	and a measure $m$ in $\mathcal{B}$ defined by $m(E)=\sum_{i\geq 0}\delta_{x_i}(E)$. Note that $m$ is $\sigma$-finite with $m\left(\Sigma_{A}\right)=\infty$. It is easy to see that $m$ is a $\sigma$-invariant $\phi$-\emph{DLR} measure. Thus, by Theorem \ref{por1G}, $m$ is a $(\phi,1)$-conformal measure.
	\end{exem}
	
It is known that when $|S|<\infty$ and the potential is Walters, the probability DLR measure is unique, see \cite{ANS}. The following result shows the uniqueness of the DLR measure when the topological Markov shift is non-compact and  satisfying the BIP property.  Example \ref{exem4} gives us a counter-example of nonuniqueness when the Markov shift does not satisfy the BIP property.

\begin{teo}\label{dlrbip}
Let $\Sigma_{A}$ be a topologically mixing Markov shift satifying BIP property and let $\phi:\Sigma_{A}\to \bbr$ be a potential satisfying the Walters condition,  $\Var_1\phi<\infty$ and $P_G(\phi)<\infty$. Then $\phi$ has a unique \emph{DLR} probability measure. In this case, the set of conformal probability measures and the set of \emph{DLR} probability measures coincide.
\end{teo}

\begin{proof}
Note that, by our assumptions and Proposition 3.8 in \cite{Sa5}, the potential $\phi$ is positively recurrent.	 
By Generalized Ruelle-Perron-Frobenius Theorem for topological mixing satisfying BIP property, there exist a non-singular probability measure $\nu$ and a continuous function $h:\Sigma_{A}\to \bbr$ such that $h$ is uniformly bounded away from zero and infinity, $L_{\phi}h=\lambda h$ and $L^*_{\phi}\nu=\lambda\nu$, where $\lambda=e^{P_{G}(\phi)}$.

Let $\mu$ be a $\phi$-DLR measure. We claim $\mu=\nu$. Note that it is enough to show $\mu \ll \nu$. In fact, by the same argument for Theorem 3.6 of \cite{Sa5}, we know that $\varphi:=\frac{\d\mu}{\d\nu}$ is a constant function equals to 1 $\nu$-a.e.

Let $n\in \mathbb{N}$ and $[\underline{a}]$ be a cylinder of length $n$. Define $M:=\sup_{n\ge 1} \Var_{n+1}\phi_{n}  + \Var_1\phi$. Then
\begin{equation*}
|\phi_n(\xi)-\phi_n(\eta)|\leq M \quad \text{for every }\xi,\eta\in[\underline{a}].
\end{equation*}
For a fixed $x\in \Sigma_{A}$, we have	
\begin{equation*}
e^{\phi_n(\underline{a}\sigma^nx)}\mathbbm{1}_{\{\underline{a}\sigma^nx\in \Sigma_{A}\}}\leq e^Me^{\phi_n(\underline{a}y)} \qquad\text{for every } y\in\sigma^n[a].
\end{equation*}
	Integrating with respect to the measure $\nu$,	
	\begin{eqnarray*}
		\int_{\sigma^n[\underline{a}]}e^{\phi_n(\underline{a}\sigma^nx)}\mathbbm{1}_{\{\underline{a}\sigma^nx\in \Sigma_{A}\}}\d\nu(y)
		\leq e^M\int_{\Sigma_{A}} L_{\phi}^n\mathbbm{1}_{[\underline{a}]}(y)\d\nu(y) 
		=e^M\lambda^n\nu\left([\underline{a}]\right).
	\end{eqnarray*}
Therefore, 
	\begin{eqnarray}\label{reg01}
	\lambda^{-n}e^{\phi_n(\underline{a}\sigma^nx)}\mathbbm{1}_{\{\underline{a}\sigma^nx\in \Sigma_{A}\}}\leq\frac{1}{\nu\big(\sigma^n[\underline{a}]\big)} e^M\nu\left([\underline{a}]\right) \qquad \text{for every } x\in \Sigma_{A}.
	\end{eqnarray}
	By BIP property, there exist $b_1,b_2,\ldots,b_N \in S$ such that
	\be\label{medidabip}
	\nu\big(\sigma^n[\underline{a}]\big)\geq \min\left\{\nu([b_i]):~i=1,\ldots, N\right\}:=K>0.
	\ee
	Replacing Inequality (\ref{medidabip}) in (\ref{reg01}), we obtain
	\begin{equation}\label{reg02}
	\lambda^{-n}e^{\phi_n(\underline{a}\sigma^nx)}\mathbbm{1}_{\{\underline{a}\sigma^nx\in \Sigma_{A}\}}\leq\frac{e^M}{K}\nu\left([\underline{a}]\right) \qquad \text{for every } x\in \Sigma_{A}.
	\end{equation} 
	Let $H_1$ and $H_2$ be positive real-valued numbers satisfying $H_1\le h\le H_2$. We have the following bounds for every $n\ge 1$,
	\begin{equation}\label{reg03}
	\frac{H_1}{H_2}\leq\lambda^{-n}\sum_{\sigma^ny=\sigma^nx}e^{\phi_n(y)}\leq\frac{H_2}{H_1}
	\end{equation}	
for every $x\in \Sigma_{A}$.	

From Inequalities $(\ref{reg02})$ and $(\ref{reg03})$ we obtain that, for every $n\ge 1$ and every cylinder $[\underline{a}]$ of length $n$, 
\be\label{dlrconforme}
\frac{e^{\phi_n(\underline{a}\sigma^nx)}\mathbbm{1}_{\{\underline{a}\sigma^nx\in \Sigma_{A}\}}}{\displaystyle\sum_{\sigma^ny=\sigma^nx}e^{\phi_n(y)}}\leq C \nu\left([\underline{a}]\right) \quad \text{for every }x\in \Sigma_{A},
\ee
where $C=\frac{H_1e^M}{H_2K}>0$. Since $\mu$ is $\phi$-DLR, integrating Inequality (\ref{dlrconforme}) with respect to $\mu$, we have $\mu\left([\underline{a}]\right)\leq C\nu\left([\underline{a}]\right)$ for every cylinder $[\underline{a}]$ of length $n$. Since $n$ is an arbitrary number, we conclude $\mu\ll\nu$.
\end{proof}

There are topological Markov shift that does not satisfy the BIP property and having more than one DLR measure associated to the same potential, as we will see below. Also the following example shows a particular topological Markov shift in which there is a DLR measure which is not a conformal measure. This example was based on Example A.1 of \cite{Sw}.

\begin{exem}\label{exem4} Consider the topological Markov shift $\Sigma_{A}$ defined by the graph of Figure \ref{fig:renin}. Let $\phi:\Sigma_{A}\to \bbr$ be a potential satisfying the Walters condition, $\Var_1\phi<\infty$ and $P_G(\phi)<\infty$. Note that~$\Sigma_{A}$ does not satisfy the BIP property.

\begin{figure}[!htb]
	\begin{center}
		\begin{tikzpicture}[scale=1.5]
		\draw   (0,0) -- (7,0);
		\node[circle, draw=black, fill=white, inner sep=1pt,minimum size=5pt] (1) at (0,0) {1};
		\node[circle, draw=black, fill=white, inner sep=1pt,minimum size=5pt] (2) at (1,0) {2};
		\node[circle, draw=black, fill=white, inner sep=1pt,minimum size=5pt] (3) at (2,0) {3};
		\node[circle, draw=black, fill=white, inner sep=1pt,minimum size=5pt] (4) at (3,0) {4};
		\node[circle, draw=black, fill=white, inner sep=1pt,minimum size=5pt] (5) at (4,0) {5};
		\node[circle, draw=black, fill=white, inner sep=1pt,minimum size=5pt] (6) at (5,0) {6};
		\node[circle, draw=black, fill=white, inner sep=1pt,minimum size=5pt] (7) at (6,0) {7};
		\node (8) at (7.3,0) {};
		\draw[->, >=stealth] (1)  to (2);
		\draw[->, >=stealth] (2)  to (3);
		\draw[->, >=stealth] (3)  to (4);
		\draw[->, >=stealth] (4)  to (5);
		\draw[->, >=stealth] (5)  to (6);
		\draw[->, >=stealth] (6)  to (7);
		\draw[->, >=stealth] (7)  to (8);
		\node [minimum size=10pt,inner sep=0pt,outer sep=0pt] {} edge [in=200,out=100,loop, >=stealth] (0);
		%\draw[->, >=stealth] (1)  to [out=90,in=90, looseness=1] (1);
		\draw[->, >=stealth] (2)  to [out=90,in=90, looseness=1] (1);
		\draw[->, >=stealth] (3)  to [out=90,in=90, looseness=1] (1);
		\draw[->, >=stealth] (4)  to [out=90,in=90, looseness=1] (1);
		\draw[->, >=stealth] (5)  to [out=90,in=90, looseness=1] (1);
		\draw[->, >=stealth] (6)  to [out=90,in=90, looseness=1] (1);
		\draw[->, >=stealth] (7)  to [out=90,in=90, looseness=1] (1);
		\end{tikzpicture}
	\end{center}
	\caption{}\label{fig:renin}
\end{figure}

Let  $\overline{x}=(1,2,3,\ldots)\in \Sigma_{A}$, and let us consider the probability measure $\mu:=\delta_{\overline{x}}$. It is easy to see that $\mu$ is $\phi$-\emph{DLR} measure. Since $\mu$ satisfies $\mu(\sigma^{-1}[1] )=0$ and $\mu\circledcirc\sigma(\sigma^{-1}[1])>0$, we have
\begin{equation*}
\mu\big(\sigma^{-1}[1] \big)\neq \lambda^{-1}\int_{\sigma^{-1}[1]}e^{\phi}\d\mu\circledcirc\sigma
\end{equation*}
for every $\lambda>0$. Therefore, there is no $\lambda>0$ such that $\mu$ is $(\phi,\lambda)-$conformal.
\end{exem}

\begin{note}\label{Remark2} The study of the existence of conservative conformal measures for topological Markov shift spaces is determined by the potential recurrence modes, see Theorem \ref{teo-RPF} and Theorem \ref{RPFt}.As we saw in Example \ref{exem4}, independently of the recurrence mode of the potential, there exists a DLR measure. We conclude that Example \ref{exem4} is an example where the set of \emph{DLR} probability measures is strictly larger than the set of conformal probability measures. In addition,  $\delta_{\overline{x}}$ is a probability DLR measure for any potential.
\end{note}

% % % % % % % % % % % % % % % % % % % % % % % % % % % % % % % % % % % % % % % %
% % % % % % % % % % % % % % % % % % % % % % % % % % % % % % % % % % % % % % % % 

We will finish this section by studying the connection between $\sigma$-invariant DLR measures and equilibrium measures. These two notions of measures are widely studied for the space $S^{\mathbb{Z}^d}$ and where it is known that they are equivalent, see \cite{Do2,Ke,Ki,LaRu,Mu}. 

\begin{prop}\label{prop-dlrinv}
Let $\Sigma_{A}$ be a topologically mixing Markov shift and $\phi:\Sigma_{A}\to\bbr$ be a recurrent potential satisfying the Walters condition and $P_G(\phi)<\infty$. Consider $m:=h\d\nu$ where $h$ is a positive continuous function and $\nu$ is a measure such that $L_{\phi}h=e^{P_G(\phi)}h$ and $L^{*}_{\phi}\nu=e^{P_G(\phi)}\nu$. Then $m$ is $\phi$-\emph{DLR} if, and only if, $h$ is constant $\nu$-a.s.
\end{prop}

\begin{proof}
Let $\widetilde{\phi}:=\phi+\log h-\log h\circ \sigma -P_G(\phi)$. It is easy to verify that
$$\int f\d m=\int L_{\widetilde{\phi}}f\d m,\quad\mbox{for all }f\in L^1(m),$$
thus $m$ is a $(\widetilde{\phi},1)$-conformal measure. Since $\|L_{\widetilde{\phi}}\mathbbm{1}\|_{\infty}<\infty$, by Theorem \ref{con-dlr} we conclude that $m$ is a $\widetilde{\phi}$-DLR measure. 

Assume that $m$ is a $\phi$-DLR measure. For every $n\geq 2$ and $c\in S$,  consider two words of size $n$ $\underline{a}:=a_0,a_1,\ldots,a_{n-1}$ and $\underline{b}:=b_0,b_1,\ldots,b_{n-1}$  such that $a_{n-1}=b_{n-1}=c$. Since, by Proposition \ref{resu2}, $m$ is $\widetilde{\phi}$-DLR,
\be\label{m1}
m\circ v_{\underline{a},\underline{b}}\big(E\big)=\int_Ee^{\phi_n(\underline{b}x_n^{\infty})-\phi_n(\underline{a}x_n^{\infty})}\d m(x),
\ee
\be\label{m2}
m\circ v_{\underline{a},\underline{b}}\big(E\big)=\int_Ee^{\widetilde{\phi}_n(\underline{b}x_n^{\infty})-\widetilde{\phi}_n(\underline{a}x_n^{\infty})}\d m(x)
\ee
for every $E\in\calb\cap[\underline{a}]$. By $(\ref{m1})$ and $(\ref{m2})$,
$$\int_Ee^{\phi_n(\underline{b}x_n^{\infty})-\phi_n(\underline{a}x_n^{\infty})}\bigg(1-\frac{h(\underline{b}x_n^{\infty})}{h(\underline{a}x_n^{\infty})}\bigg)\d m(x)=0$$
for every $E\in\calb\cap[\underline{a}]$. By the continuity of the function $h$ we have $h(\underline{a}x_n^{\infty})=h(\underline{b}x_n^{\infty})$ for all $x\in[\underline{a}]$ such that $a_{n-1}=b_{n-1}=c$.

Since $c\in S$, $\underline{a}$ and $\underline{b}$ were chosen arbitrarily such that $a_{n-1}=b_{n-1}=c$, we conclude that $h$ is $\sigma^{-n}\calb$-measurable, and so $h$ is $\bigcap_{n\geq 1}\sigma^{-n}\calb$-measurable. By Theorem 2.5 in \cite{Sa5} we know that $\nu$ is exact, concluding that $h$ is constant $\nu$-a.s.

Now let us prove the converse. Suppose $h(x)=\alpha$, $\nu$-a.s., for some $\alpha\in\mathbb{R}$. Since $\nu$ is $(\phi,e^{P_G(\phi)})$-conformal and $m=\alpha\nu$, we have that $m$ is also $(\phi,e^{P_G(\phi)})$-conformal. Note that
$$m\big(\pi^{-1}_n\{a\}\big)=m([a])<\infty\quad \text{for all } a\in S,~ n\in\bbn,$$
implying that the sub-$\sigma$-algebra $\sigma^{-n}\calb$ is compatible with $m$ for every $n\geq1$. By Theorem \ref{con-dlr}, $m$ is a $\phi$-DLR measure.
\end{proof}

\begin{teo}\label{teo-dlrinv}
Let $\Sigma_{A}$ be a topologically mixing Markov shift and let $\phi:\Sigma_{A}\to \bbr$ be a potential satisfying the Walters condition and $P_G(\phi)<\infty$. Then:
\begin{itemize}
\item[i)] If $\Sigma_{A}$ has the BIP property and  $\Var_1\phi<\infty$ then
 $$
 \{\mu \in \mathcal{M}^1_{\sigma}(\Sigma_A): \mu~\mbox{is}~\phi\mbox{-}\emph{DLR}~\mbox{and}~h_{\mu}(\sigma)<\infty\}\subseteq\{\mu\in \mathcal{M}^1_{\sigma}(\Sigma_A):\mu~\mbox{is}~\phi\mbox{-}\mbox{equilibrium}\}.
 $$
\item[ii)] If $\sup \phi<\infty$ then
$$
\{\mu \in \mathcal{M}^1_{\sigma}(\Sigma_A):\mu~\mbox{is}~\phi\mbox{-}\mbox{equilibrium}\}\subseteq\{\mu \in \mathcal{M}^1_{\sigma}(\Sigma_A):\mu~\mbox{is}~\widetilde{\phi}\mbox{-}\emph{DLR}\},
$$
\end{itemize}
where $\widetilde{\phi}=\phi+\log h-\log h\circ\sigma- P_G(\phi)$.
\end{teo}
\begin{proof} Let us prove item i). By Lemma $4$ in \cite{BMP} we have $\sup\phi<\infty$. By Theorem $1$ in \cite{Sa4} the potential $\phi$ is positive recurrent, let $m=h\d\nu$, where $h$ is a positive continuous function and $\nu$ is a finite measure such that $L_{\phi}h=e^{P_G(\phi)}h$ and $L^{*}_{\phi}\nu=e^{P_G(\phi)}\nu$. Consider $\mu$ be a  $\sigma$-invariant $\phi$-DLR measure such that $h_{\mu}(\sigma)<\infty$. Let $n\in\bbn$ and $[\underline{a}]$ be a cylinder of length $n$. Let $H_1,H_2>0$ the positive constants such that $H_1<h(x)<H_2$ for every $x\in\Sigma_{A}$ and $M:=\sup_{n\ge 1} \Var_{n+1}\phi_{n}  + \Var_1\phi$. By the same argument in the proof of Theorem \ref{dlrbip}, we have
\begin{equation}\label{reg02-e}
	\lambda^{-n}e^{\phi_n(\underline{a}\sigma^nx)}\mathbbm{1}_{\{\underline{a}\sigma^nx\in {\Sigma_{A}}\}}\leq\frac{e^M}{H_1K}m\left([\underline{a}]\right)\quad \text{for all } x\in {\Sigma_{A}},
	\end{equation} 
	where $K>0$ is given by $(\ref{medidabip})$. By $(\ref{reg03})$ and $(\ref{reg02-e})$, 
%given $n\in\bbn$ and $[\underline{a}]$ a cylinder of size $n$,
\be\label{e-c}
\frac{e^{\phi_n(\underline{a}\sigma^nx)}\mathbbm{1}_{\{\underline{a}\sigma^nx\in {\Sigma_{A}}\}}}{\displaystyle\sum_{\sigma^ny=\sigma^nx}e^{\phi_n(y)}}\leq Cm\left([\underline{a}]\right)\quad \text{for all } x\in {\Sigma_{A}},
\ee
where $C=\frac{e^M}{H_2K}$. Since $\mu$ is $\phi$-DLR, integrating (\ref{e-c}) in both sides with respect to $\mu$, we have $\mu\left([\underline{a}]\right)\leq Cm\left([\underline{a}]\right)$ for every cylinder $[\underline{a}]$ of length $n$, concluding $\mu\ll m$. Since $m$ is an ergodic measure, by Theorem $4.7$ in \cite{Sa5}, we conclude $\mu=m$.

To prove item ii), note that the equilibrium measure, when does exist, is given by $m=h\d\nu$. A standard calculation shows $L^{*}_{\widetilde{\phi}}m=m$, concluding that $m$ is a $(\widetilde{\phi},1)$-conformal measure and therefore it is a $\widetilde{\phi}$-DLR measure.
\end{proof}

\begin{note}\label{note-equi-infi} When the potential $\phi$ is null recurrent $\mu=h\d\nu$ is an infinite invariant measure, with $h$ and $\nu$ as in Theorem \ref{teo-RPF}. The measure $\mu$ is the only conservative, ergodic and invariant measure that satisfies the relation $h_{\mu}(\sigma)=\mu\left(P_G(\phi)-\phi\right)$ for the class of weakly H\"{o}lder continuous potentials, see Theorem $2$ in \cite{Sa2}. This is the reason why O. Sarig proposed a notion of ``infinite equilibrium measure". It is worth noting that item $ii)$ of Theorem \ref{teo-dlrinv} is also valid for equilibrium measures and infinite \emph{DLR} measures, i.e., 
	$$
	\left\{\mu \in \mathcal{M}_{\sigma}(\Sigma_A):\mu~\mbox{is}~\phi\mbox{-}\mbox{equilibrium}\right\}\subseteq\big\{\mu \in \mathcal{M}_{\sigma}(\Sigma_A):\mu~\mbox{is}~\widetilde{\phi}\mbox{-}\emph{DLR}\big\}.
	$$
\end{note} 

%\begin{note}\label{Remark3}\textbf{\textsc{ Infinite Equilibrium Measures.}} Sarig proposed a notion of infinite equilibrium measure in \cite{Sa2}. We note that item $ii.)$ of the previous theorem is valid under more general conditions, see the Corollary \ref{coro-Cyr}.
%\end{note} 

% % % % % % % % % % % % % % % % % % % % % % % % % % % % % % % % % % % % % % % % % % % % % % % %
% % % % % % % % % % % % % % % % % % % % % % % % % % % % % % % % % % % % % % % % % % % % % % % %

\section{Volume-Type Phase Transitions for Renewal shifts}\label{sec:phase-transition}

It is usual in the literature of statistical mechanics to study models with a good behavior respect to phase transitions, that is, models such that there exist a parameter (temperature is an example) for which you have a unique point separating different situations. Even for general regular potentials and interactions, we have results like the \textit{Dobrushin uniqueness theorem} \cite{Do3, Geo, Sim, FV}, which one of the implications for finite state-space systems for high enough temperatures is to guarantee that does exist exactly one DLR state. This fact motivates the usual definition of \textit{phase transition} used by researchers in probability. In statistical mechanics, it is common to define that a model presents a \textit{phase transition} when, for low enough temperatures, there exists more than one DLR measure. So, in this case, the transition is from one to several DLR measures, ferromagnetic systems like Ising type models are examples of models for this situation.

In the setting of countable Markov shifts, the following result given by O. Sarig shows that the renewal shift has a good behavior respect the phase transition according to the recurrence modes. Thanks to the generalized Perron-Frobenius Theorem \ref{teo-RPF}, we know that the next theorem is equivalent to say that in high temperatures there exist an  equilibrium measure for the potential and, after a critical beta $\beta_c$, we have the absence of equilibrium probability measures and the pressure is a linear function. When the pressure is a linear function at low temperatures, this fact is called \textit{freezing phase transition}.

\begin{teo}[\cite{Sa3}]\label{SarTF}
Let $\Sigma_A$ be the renewal shift and let $\phi:\Sigma_A\to \mathbb{R}$ be a weakly H\"{o}lder continuous function such that $\sup \phi<\infty$. Then there exists $0<\beta_c\le \infty$ such that:
\begin{enumerate}
\item[i)] For $\beta\in(0,\beta_c)$, the potential $\beta \phi$ is positive recurrent, and for $\beta>\beta_c$, $\phi$ is transient.
\item[ii)] $P_G(\beta\phi)$ is real analytic in $(0,\beta_c)$ and linear in $(\beta_c,\infty)$. Moreover, $P_G(\beta\phi)$ is continuous

 \hspace{-0.55cm}  at $\beta_c$ but not analytic.
\end{enumerate} 
\end{teo}

From now on $\beta_c$ will denote the critical value given by the Theorem \ref{SarTF}. Note that  for every $\beta\in(0,\beta_c)$ there exists a $(\beta\phi,e^{P_G(\beta\phi)})$-conformal measure. The eigenmeasure associated to the potential $\beta \phi$ will be denoted by $\nu_{\beta}$.

Let us define the function 
\begin{equation*}
m(\phi):=\sup_{\mu\in\mathcal{M}^{1}_{\sigma}(\Sigma_A)}\int\phi\d\mu.
\end{equation*}
A measure $\mu_{\phi}\in\mathcal{M}_{\sigma}^1(\Sigma_{A})$ is called \emph{$\phi$-maximizing} if $m(\phi)=\int\phi\d\mu_{\phi}$.

Under the conditions of Theorem  \ref{SarTF}, for each $\beta\geq 0$,
$$P_G(\beta\phi)\leq\log2+{\beta\sup\phi},$$
consequently $P_G(\beta\phi)<\infty$. Let us define the function $\psi:(0,\beta_c)\to\mathbb{R}$ given by
\begin{equation*}
\psi(\beta):=\frac{P_G(\beta\phi)}{\beta}.
\end{equation*}

\begin{prop}\label{P} Let $\Sigma_A$ be the renewal shift and let $\phi:\Sigma_A\to \mathbb{R}$ be a weakly H\"{o}lder continuous function such that $\sup \phi<\infty$. Then the function $\psi$ is continuous, strictly decreasing, and $\lim_{\beta\to0^{+}}\psi(\beta)=+\infty$.
\end{prop}

\begin{proof}
By Theorem \ref{SarTF} item $ii)$, the function $\psi$ is continuous in $(0,\beta_c)$. For $\beta\in(0,\beta_c)$, let $\mu_{\beta}$ be the unique equilibrium measure associated to the potential $\beta\phi$. Define the function $\varphi:(0,\beta_c)\to \mathbb{R}$ by
 \begin{equation*} %\label{estrito}
 \varphi(\beta):=P_G(\beta\phi)-\beta m(\phi).
 \end{equation*}
%Consider $0<\beta_1<\beta<\beta_c$,
 %By Lemma $9$ in \cite{BMP}, the function $\varphi$ is non-increasing.
%\be\label{presão}
%P_G(\beta_1\phi)-\beta_1m(\phi)\geq P_G(\beta\phi)-\beta m(\phi)\geq0.
%\ee
%Next 
Note that $\varphi$ is positive for $\beta_c<\infty$ by Theorem $1.1$ in \cite{Io} and Theorem $2$ in \cite{Sa3}. Now let us consider  $\beta_c=\infty$. By Theorem $1.1$ in \cite{Io}, there exists a $\phi$-maximizing measure $\mu_{\phi}\in\mathcal{M}_{\sigma}^1(\Sigma_{A})$. In this case
 \begin{equation*} %\label{posi}
 \varphi\geq 0, \quad\mbox{for every }\beta>0,
 \end{equation*}
because $\varphi(\beta)\geq h_{\mu_{\phi}}\geq 0$.  Note that the function $\varphi$ is convex. Since $\frac{\d}{\d\beta} P(\beta \phi)=\int\phi\d\mu_{\beta}$ for $\beta>0$, by Proposition 2.6.13 in \cite{MaUr2},
 \begin{equation}\label{deri}
  \frac{\d}{\d\beta}\varphi(\beta)=\int\phi\d\mu_{\beta}-m(\phi)\le 0.
 \end{equation}
Thus, $\varphi$ is non-increasing.

 Suppose, by contradiction,
 % that $(\ref{estrito})$ is false, i.e., by (\ref{posi}), 
that there exists $\tilde{\beta}>0$ such that $P_G(\tilde{\beta} \phi)=\tilde{\beta}m(\phi)$. Then
  \begin{eqnarray}\label{cero}
  P_G(\beta \phi)={\beta}m(\phi), \quad\mbox{for every } \beta\in[\tilde{\beta},\infty).
  \end{eqnarray}
By $(\ref{deri})$,  every measure $\mu_{\beta}$ is $\phi$-maximizing for $\beta>\tilde{\beta}$. Thus $h_{\mu_{\phi}}(\sigma)=0$ and $h_{\mu_{\beta}}(\sigma)=0$ for every $\beta>\tilde{\beta}$. Note that, for every $\beta>\tilde{\beta}$, both measures $\mu_{\phi}$ and $\mu_{\beta}$ are equilibrium measures for the potential $\beta\phi$. Since the equilibrium measure is unique, we have that $\mu_{\phi}=\mu_{\beta}$ for every $\beta>\tilde{\beta}$. Note that this argument also conclude that there is only one maximizing measure.
  
  Consider $\beta_1,\beta_2\in(\tilde{\beta},+\infty)$. Since $\mu_{\beta_1\phi}=\mu_{\beta_2\phi}$, then $\beta_1\phi\sim\beta_2\phi+c$ for some $c\in\mathbb{R}$, see Theorem $4.8$ in \cite{Sa5}. Since $P_G(\beta_1\phi)=P_G(\beta_2\phi)+c$, by $(\ref{cero})$,
  \begin{equation}\label{cons1}
  c=(\beta_1-\beta_2)m(\phi).
  \end{equation}
  Consider $x=(1,1,1,\ldots)\in\Sigma_{A}$. There exists a function $\alpha$ such that $\beta_1\phi=\beta_2\phi+\alpha-\alpha\circ\sigma+c$. Thus,
  \begin{equation}\label{cons2}
  c=(\beta_1-\beta_2)\phi(x).
  \end{equation}
  By $(\ref{cons1})$ and $(\ref{cons2})$, 
  $$m(\phi)=\phi(x)=\int\phi\d\delta_{x},$$
  then $\delta_{x}$ is a maximizing measure, and therefore a equilibrium measure, a contradiction. 
  
Since $\varphi$ is positive and non-increasing, for every $\beta_1< \beta_2$, we have $\varphi(\beta_1)\ge \varphi(\beta_2)>0$, thus
\begin{equation*}
\psi(\beta_1)=\frac{\varphi(\beta_1)}{\beta_1}+m(\varphi)> \frac{\varphi(\beta_2)}{\beta_2}+m(\varphi) = \psi(\beta_2),
\end{equation*}
concluding that $\psi$ is decreasing.

%By $(\ref{presão})$ and the fact that $0<\beta_1<\beta$ we have $\psi(\beta_1)>\psi(\beta)$, therefore $\psi$ is decreasing.
 
It remains to show $\lim\limits_{\beta\to{0^+}}\psi(\beta)=+\infty$. Since $\psi$ is decreasing, we have $\psi'(\beta)<0$, which implies $P_G^{\prime}(\beta\phi)-\psi(\beta)<0$.
Let $\beta_0\in (0,\beta_c)$. Since $P_G(\beta\phi)$ is convex with respect to $\beta$,
%\be
%P_G(\beta_0\phi)+P_G^{\prime}(\beta_0\phi)\beta-P_G^{\prime}(\beta_0\phi)\beta_0\leq P_G(\beta\phi),
%\ee
%equivalently
\begin{equation*}
P_G^{\prime}(\beta_0\phi)+\frac{\beta_0}{\beta}\left(\psi(\beta_0)-P_G^{\prime}(\beta_0\phi)\right)\leq \psi(\beta).
\end{equation*}
Thus, the left hand side diverges to infinity when $\beta$ converges to $0^+$.
\end{proof}

\begin{teo}\label{finitude} Let $\Sigma_A$ be the renewal shift and let $\phi:\Sigma_A\to \mathbb{R}$ be a weakly H\"{o}lder continuous function such that $\sup \phi<\infty$. Then, there exists $\tilde{\beta}_c\in(0,\beta_c]$ such that the eigenmeasure $\nu_{\beta}$ is finite for $\beta\in(0,\tilde{\beta}_c)$, and $\nu_{\beta}$ is infinite for $\beta\in(\tilde{\beta_c},\beta_c)$. Moreover, if $\Var_1\phi<\infty$ then $\nu_{\beta}$ is finite for all $\beta\in(0,\beta_c)$.
\end{teo}

\begin{proof} 
 For each $\beta\in(0,\beta_c)$, the potential $\beta \phi$ is positively recurrent. Thus, by Theorem \ref{teo-RPF}, there exists a $\sigma$-finite measure $\nu_{\beta}$ such that
\begin{equation}\label{FormulaRe}
e^{P_G(\beta\phi)}\int f\d\nu_{\beta}=\int L_{\beta\phi}f\d\nu_{\beta},\quad\mbox{ for every }f\in L^1\left(\nu_{\beta}\right).
\end{equation}
Consider $\beta\in(0,\beta_c)$. For each $a\geq 2$, we consider $f:=\mathbbm{1}_{[a]}$. Thus, by Equation (\ref{FormulaRe}),
\begin{equation}\label{FormuRen}
e^{P_G(\beta\phi)}\nu_{\beta}([a])=\int_{[a-1]} e^{\beta\phi(ax)}\d\nu_{\beta}.
\end{equation}
For each $n\ge a$, consider the periodic orbit $$\gamma_a^n:=(\overline{a,a-1,\ldots,1,n,n-1,\ldots,a+1}).$$
Then
\begin{eqnarray}\label{ineqRe}
\phi(\gamma_a^n)-\Var_a\phi\leq\phi(ax)\leq\phi(\gamma_a^n)+\Var_a\phi,\quad\mbox{ for every }x\in\sigma[a].
\end{eqnarray}
Substituting (\ref{ineqRe}) in  Equation (\ref{FormuRen}),
\begin{eqnarray}\label{RenMes}
e^{-\beta\Var_a\phi+\beta\phi(\gamma_a^n)-P_G(\beta\phi)}\nu_{\beta}([a-1])\leq \nu_{\beta}([a])\leq e^{\beta \Var_a\phi+\beta\phi(\gamma_a^n)-P_G(\beta\phi)}\nu_{\beta}([a-1]).
\end{eqnarray}
Iterating (\ref{RenMes}) from $a=1$ to $n$,
$$\nu_{\beta}\left([n]\right)\leq e^{\beta\sum_{j=2}^n\left(\phi(\gamma_{j}^n)+\Var_{j}\phi\right)}e^{-(n-1)P_G(\beta\phi)}\nu_{\beta}\left([1]\right).$$
Therefore
\begin{eqnarray*}
\nu_{\beta}\left(\Sigma_{A}\right)\leq\nu_{\beta}([1])\left(\sum_{n\geq 1}e^{\beta\sum_{j=2}^n\Var_{j}\phi+\phi(\gamma_{j}^n)}e^{-(n-1)P_G(\beta\phi)}\right),
\end{eqnarray*}
concluding that $\nu_{\beta}$ is finite when the series
\begin{equation*} %\label{series}
\sum_{n\geq 1}e^{\beta\sum_{j=2}^n\Var_{j}\phi+\phi(\gamma_{j}^n)}e^{-(n-1)P_G(\beta\phi)}
\end{equation*}
converges, i.e.,
\begin{equation*}
\limsup_{n\to\infty}\frac{1}{n}\sum_{j=2}^n\phi\left(\gamma_{j}^n\right)<\frac{P_G(\beta\phi)}{\beta}.
\end{equation*}

Analogously, by $(\ref{RenMes})$, we have that $\nu_{\beta}$ is infinite when
\begin{equation*}
\limsup_{n\to\infty}\frac{1}{n}\sum_{j=2}^n\phi\left(\gamma_{j}^n\right)>\frac{P_G(\beta\phi)}{\beta}.
\end{equation*}

Let us define
\be\label{bet}
\tilde{\beta}_c:= \sup\left\{ \beta\in (0,\beta_c]: \limsup_{n\to\infty}\frac{1}{n}\sum_{j=2}^n\phi\left(\gamma_{j}^n\right)<\frac{P_G(\beta\phi)}{\beta} \right\}.
\ee
By Proposition \ref{P}, $\tilde{\beta}_c$ exists and it is positive.

Now let us consider $\Var_1\phi<\infty$ and, for every $n\geq 2$, we define $x_n:=(\overline{1,n,n-1\ldots,2})$ and $x_1=(1,1,\ldots,1,\ldots)$. Note that 
 \begin{equation*} %\label{V11}
 \frac{1}{n}\sum_{j=2}^n\phi\left(\gamma_{j}^n\right)=\frac{\phi_n(x_n)}{n}-\frac{\phi(x_n)}{n},
 \end{equation*}
 and using the fact that $\Var_1\phi<\infty$, we have
 \begin{equation*} %\label{V12}
 \lim_{n\to\infty}\frac{\phi(x_n)}{n}=0.
 \end{equation*}
 Therefore $\tilde{\beta}_c$ can be written as
 \be\label{V1c}
  \tilde{\beta_c}= \sup\left\{ \beta\in (0,\beta_c]: \limsup_{n\to\infty}\frac{\phi_n(x_n)}{n}<\frac{P_G(\beta\phi)}{\beta} \right\}.
  \ee

By the Discriminant Theorem, Theorem $2$ in \cite{Sa3}, we have
 $$\limsup_{n\to\infty}\frac{\phi_n(x_n)}{n}<\frac{P_G(\beta\phi)}{\beta}$$
 for every $\beta\in(0,\beta_c)$. Then, by $(\ref{V1c})$, we conclude that $\beta_c=\tilde{\beta_c}$, i.e., $\nu_{\beta}$ is finite for every $\beta\in(0,\beta_c)$.
\end{proof}

Note that in general the critical values $\beta_c$ and $\tilde{\beta_c}$ are different, but they can coincide. For instance, when the potential is a constant function, it is easy to see that $\tilde{\beta_c}={\beta_c}=+\infty$.

\begin{exem} Consider de potential  $\phi(x)=x_0-x_1$. We have $\tilde{\beta_c}=\log2$ and ${\beta_c}=+\infty$. 
\end{exem}

\begin{note}
Note that log 2 is precisely the topological entropy $h$ (which coincides with the Gurevich entropy of the graph) of the renewal shift. Thomsen proved the existence of $\beta$-KMS weights for $\beta > h$ on certain graph $C^{*}$-algebras in \cite{Tho1}. For the moment, this is just an example. We do not know if the critical point of the volume-type phase transition can be characterized in terms of some entropy or another thermodynamic quantity. 
\end{note}

%\textit{Remark 4}: 
Let $\{d_i\}_{i\geq1}$ be an increasing sequence of positive integers. We denote by $\calr_1$ the transition matrix $(A_{i,j})_{\bbn\times \bbn}$ with entries $A(1,1), A(i+1,i),A(1,d_i)$ equal to one for all $i\geq 1$, and the other entries equal to zero. Note that when $d_i=i$ for every $i\ge 1$, the $\Sigma_{\calr_1}$ is the renewal shift. Theorem \ref{SarTF} and Theorem \ref{finitude} hold for $\Sigma_{\calr_1}$ with the same $\tilde{\beta}_c$ given by (\ref{bet}).
%\bfblue{Adicionar mais informac\~{a}o sobre $\calr_1$ e colocar uma referencia relacionada a $\calr_1$.}

Consider $\Sigma_{\calr^{-}}$ be the topological Markov shift such that $\calr^{-}$ is the transition matrix $(A_{i,j})_{\bbn\times \bbn}$ with entries $A(i,i+1)$ and $A(i,1)$ equal to one for every $i\geq1$ and the other entries equal to zero, see Figure 3. Theorem \ref{SarTF} also holds for $\Sigma_{\calr^{-}}$, and the proof is analogous.
 
For $\Sigma_{\calr^{-}}$, Theorem \ref{teo-RPF} and Theorem \ref{SarTF} guarantee the existence of an eigenmeasure for $\beta\in (0,\beta_c)$. Moreover, since $\Sigma_{\mathcal{R}^-}$ is locally compact, by Theorem \ref{RPFt}, we guarantee the existence of an eigenmeasure for $\beta\geq\beta_c$. 

The next proposition give us conditions to the absence of volume-type phase transitions for potentials defined in  the shift $\Sigma_{\calr^{-}}$. So far we do not have a general theorem as in the case of the standard renewal shift.

\begin{prop}\label{prop:r-}
Let $\phi:\Sigma_{\calr^{-}}\to\bbr$ be a weakly H\"{o}lder continuous function such that $\sup\phi<\infty$, $\Var_2\phi=0$ and $\phi_{n}(\gamma_n)=0$ for every $n\geq1$, where $\gamma_n=(\overline{1,2,\ldots,n})$. Then,
\begin{itemize}\label{cond-rev}
\item[i)] If $\limsup_{n\to\infty}\frac{\phi(n,1)}{n}< 0$, \text{then} $\nu_{\beta}$ \text{is finite for every} $\beta>0$.
\item[ii)] If $\limsup_{n\to\infty}\frac{\phi(n,1)}{n}> 0$, \text{then} $\nu_{\beta}$ \text{is infinite for every} $\beta>0$.
\end{itemize}
\end{prop}

\begin{proof}
For each $\beta> 0$, the Gurevich Pressure is equal to $P_G(\beta\phi)=\log2$, and the potential $\beta \phi$ is positively recurrent. Thus, by Theorem \ref{teo-RPF}, there exists a $\sigma$-finite measure $\nu_{\beta}$ such that
\begin{equation}\label{rever}
2\int f\d\nu_{\beta}=\int L_{\beta\phi}f\d\nu_{\beta},\quad\mbox{ for every }f\in L^1\left(\nu_{\beta}\right).
\end{equation}
By Equation $(\ref{rever})$, for each $n\geq 2$,
\begin{equation}\label{rev}
\nu_{\beta}([n])=e^{\beta\phi(n,1)}\nu_{\beta}([1]).
\end{equation}
Then, iterating (\ref{rev}),
\begin{equation*} 
\nu_{\beta}\left(\Sigma_{\mathcal{R
}^{-}}\right)=\nu_{\beta}([1])\sum_{n\geq 1}e^{\beta\phi(n,1)}.
\end{equation*}
Therefore $\nu_{\beta}$ is finite if $\limsup_{n\to\infty}\frac{\phi(n,1)}{n}<0$ and infinite if $\limsup_{n\to\infty}\frac{\phi(n,1)}{n}>0$ for every $\beta>0$.
\end{proof}

\begin{exem} Consider $\phi(x)=x_1-x_0$. Then, by the item i) of the Proposition \ref{prop:r-}, $\nu_{\beta}(\Sigma_{A})$ is finite for every $\beta>0$.
\end{exem}

\begin{exem}\label{exem6} Take $\phi(x)\equiv c$ with $c\in\bbr$. From Equation (\ref{rev}), it is easy to see that $\nu_{\beta}(\Sigma_{A})$ is infinite for every $\beta>0$. Note that, in this case, $\limsup_{n\to\infty}\frac{\phi(n,1)}{n}= 0$ and then the previous proposition is not sharp.
\end{exem}

\begin{note} It is important to mention that, given a potential $\phi$, even the very basic question if for a fixed $\beta >0$ all the $\beta \phi$-DLR measures give the same volume to the space $\Sigma_{A}$ is not obvious. In the previous example, for the constant potential, for each $\beta >0$, we have an infinite $\beta \phi$-DLR measure and also the probability DLR measure $\delta_{\overline{x}}$ of the Example \ref{exem4}.

\end{note}

In the next example, we present a potential that does not satisfy the conditions of the previous proposition which presents a volume-type phase transition.

\begin{exem}\label{exem7} Consider $\phi:\Sigma_{\calr^{-}}\to\bbr$ given by $\phi(x)=\log\frac{x_1}{x_0}$. Note that $\limsup_{n\to\infty}\frac{\phi(n,1)}{n}=0$ and $\nu_{\beta}$ is infinite for $\beta\leq 1$ and finite for $\beta>1$.
\end{exem}

\section{Concluding remarks}\label{sec:concluding}

We started the study of infinite DLR measures on countable Markov shifts. We explored the connection with conformal measures and the thermodynamic formalism for unidimensional systems with infinitely but countable states, a setting where the machinery of the Ruelle's operator can be applied. On the other hand, it seems that there are no results about infinite DLR measures on multidimensional subshifts from $\mathbb{N}^{\mathbb{Z}^d},$ for $d \geq 2$. Maybe a good direction to explore and go beyond the setting where the Ruelle operator is the main tool.

Another natural question is about the shifts and potentials with a well-behavior of the phase transition with respect to the volume, that is, a unique critical point that separates finite and infinite DLR measures. We proved the uniqueness of the critical point $\tilde{\beta_c}$ for the standard renewal shift, but we do not know about general results for other shifts even in the class of the renewal type shifts, see our last examples. Concrete examples with infinitely many critical points respect to the volume-type transition are not known.

Finally, the analogous objects to the DLR states for quantum models are the KMS states. In \cite{Bras}, Brascamp proved for a special class of interactions (a family of local functions that generates the potential) called \textit{classical}, that the KMS equations reduce to the DLR equations for such potentials. After this, Araki and Ion \cite{Ara} defined a new condition for equilibrium, now called \textit{Gibbs-Araki condition}, showed that when the interaction is classical, the Gibbs-Araki condition reduces to the DLR equations showing the equivalence to the KMS condition for these interactions.   

The topic of \textit{KMS weights} (the analogous notion to the infinite DLR measure in the quantum setting) has been developed on the context of $C^*$-algebras, in particular for groupoid $C^*$-algebras, see \cite{Chris, Tho1, Tho2}. Paradoxically, we are not aware of a systematic study of infinite DLR measures unless the progress made in infinite ergodic theory since Aaronson, Sarig, and others \cite{Aar, Sa5}. There is also some literature in the physics community as \cite{AKB, LeiBar}. In some cases, phase transitions (in the sense of the number of equilibrium states) in the classical setting imply phase transitions in the quantum framework, see \cite{Tho3}. So it is natural to investigate if this volume-type transition also forces a transition from KMS states to KMS weights; a reference for DLR measures on groupoids is \cite{BEFR}.

%O teorema acima nos permite definir a esperança condicional no caso de medidas $\sigma-$finitas, note que o preço a pagar é que a medida tem que ser $\sigma-$finita na sub-$\sigma$-álgebra restrita.

\section*{Acknowledgements}

The authors thank Ricardo Freire for his assistance for the proof of Proposition \ref{P}. We also thank Xuan Zhang and Van Cyr for comments about possible implications of the finiteness of the first variation, which helped us. We thank Lucas Affonso for his friendship over the years, for pointing out a gap in an earlier version of this paper, and for many references and interesting discussions about mathematical physics. We also thank Aernout van Enter for all the discussions about mathematical physics over the years since we were Ph.D. students, for all his generosity to read this manuscript and share with us comments and references about the subject of this paper. RB is supported by CNPq Grant 312294/2018-2 and FAPESP Grant 16/25053-8. ER was supported by Coordena\c{c}\~ao de Aperfei\c{c}oamento de Pessoal de N\'ivel Superior - Brasil (CAPES) and CNPq.

\end{document}